\documentclass{amsart}

\usepackage{amsfonts, amsmath,amscd, amssymb, latexsym, mathrsfs, stmaryrd, verbatim, wasysym }
\usepackage{enumerate}
\usepackage{graphicx}
\usepackage[all]{xy}
\newtheorem{theorem}{Theorem}

\newtheorem{definition}{Definition}

\newtheorem{corollary}[theorem]{Corollary}
\newtheorem{lemma}[theorem]{Lemma}
\newtheorem{proposition}[theorem]{Proposition}
\theoremstyle{definition}
\newtheorem{remark}[theorem]{Remark}
\newtheorem{example}[theorem]{Example}

\newtheorem{assumption}[theorem]{Assumption}
\numberwithin{equation}{section}
\numberwithin{theorem}{section}
\numberwithin{definition}{section}

 \usepackage{color}
\definecolor{darkgreen}{cmyk}{1,0,1,.2}
\definecolor{m}{rgb}{1,0.1,1}
\long\def\red#1{\textcolor {red}{#1}}
\long\def\blue#1{\textcolor {blue}{#1}}

\DeclareMathOperator{\supp}{supp}   
\DeclareMathOperator{\Hom}{Hom}    

 \DeclareMathOperator{\tr}{tr}

\DeclareMathOperator{\Ind}{Ind}

\DeclareMathOperator{\const}{const}

\newcommand{\forget}[1]{}

\def  \nuint {\raise10pt\hbox{$\nu$}\kern-6pt\int}

\newcommand\Tr{\operatorname{Tr}}

\def\N{\mathcal N}
\def \L{\mathcal L}
\def \A{\mathcal A}

\def \P{\mathcal P}

\newcommand\G{\mathcal G}
\newcommand\Q{\mathcal Q}
\newcommand\R{\mathcal R}
\newcommand\V{\mathcal V}

\def \L {{\cal L}}
\def \Sp {{\cal S}}
\newcommand\B{\mathcal B}
\newcommand\Bi{\B^\infty}
\def \J{\mathcal J}

\def \Ch {{\rm Ch}}
\def\Id{{\rm Id}}

\newcommand\cyl{\operatorname{cyl}}
\newcommand\ha{\frac12}

\newcommand\D{\mathcal D}
\newcommand\Di{D\kern-6pt/}
\newcommand\cDi{{\mathcal D}\kern-6pt/}
\newcommand\spi{S\kern-6pt/}
\newcommand \cspi{\Sp\kern-6pt/}

\newcommand\CC{\mathbb C}

\def \cal {\mathcal}

\newcommand\NN{\mathbb N}

\newcommand\RR{\mathbb R}
\newcommand\ZZ{\mathbb Z}

\newcommand\pa{\partial}
\newcommand\Ker{\operatorname{Ker}}

{\catcode`@=11\global\let\c@equation=\c@theorem}

\DeclareMathOperator{\btau}{^\mathrm{b} \hspace{0pt}\Tr_c}
\newcommand{\bint}{\sideset{^{\mathrm{b}\!\!\!}}{_{M}}\int}

\allowdisplaybreaks[2]
\date{}
 \usepackage{color}
\definecolor{darkgreen}{cmyk}{1,0,1,.2}
\definecolor{m}{rgb}{1,0.1,1}
\long\def\red#1{\textcolor {red}{#1}}
\long\def\blue#1{\textcolor {blue}{#1}}

\title[G-proper manifolds]{Higher genera for proper actions of Lie groups
}
\author{Paolo Piazza}
\address{Dipartimento di Matematica, Sapienza Universit\`a di Roma}
\email{piazza@mat.uniroma1.it}
\author{Hessel B. Posthuma}
\address{Korteweg--de Vries Institute for Mathematics, University of Amsterdam }
\email{H.B.Posthuma@uva.nl}



\begin{document}

\begin{abstract}
Let $G$ be a  Lie group with finitely many connected components and 
let $K$ be  a maximal compact subgroup.
We assume that $G$ satisfies the rapid decay (RD) property and that $G/K$ has a
non-positive sectional curvature.
As an example, we can take $G$ to be a connected semisimple Lie group. Let $M$ be a $G$-proper manifold with compact quotient $M/G$.
Building on \cite{CM} and \cite{PPT}  we establish  index formulae for the $C^*$-higher indices of a $G$-equivariant 
Dirac-type operator on $M$. We use these formulae to
investigate geometric properties of suitably defined higher genera on $M$. In particular, we establish
 the $G$-homotopy invariance of the higher signatures of a $G$-proper manifold
and  the vanishing of the $\widehat{A}$-genera of a $G$-spin $G$-proper manifold admitting
a $G$-invariant metric of positive scalar curvature.
\end{abstract}

\maketitle

\section{Introduction}\label{sect:intro}
The aim of this paper is introduce certain geometric invariants associated to proper actions of Lie groups, generalizing the (higher) signatures and $\widehat{A}$-genera. 
Let $G$ be a Lie group satisfying the following assumptions:
\begin{itemize}
\item $G$ has finitely many components.
\item Because $|\pi_0(G)|<\infty$, $G$ has a maximal compact subgroup $K$, unique up to conjugation, and we assume that the homogeneous space $G\slash K$ has non-positive sectional curvature with respect the G-invariant  metric induced by a  AdK-invariant inner product $\left<~,~\right>$ on the Lie algebra $\mathfrak{g}$.
\item $G$ satisfies the rapid decay (RD) property.
\end{itemize}

We will explain these last two hypothesis  in the course of the paper; it suffices for now to remark that natural examples of 
groups satisfying our assumptions are given by connected semisimple Lie groups. The homogeneous  space 
$G\slash K$ is a smooth model for $\underline{E}G$, the classifying space for proper actions of $G$, c.f.\ \cite{BCH}: for any smooth proper action
of $G$ on a manifold $M$, there exists a smooth $G$-equivariant classifying map $\psi_M:M\to G\slash K$, unique up to $G$-equivariant homotopy.
Assuming in addition that the action is {\em cocompact}, i.e., that the quotient $M\slash G$ is compact, we can fix a {\em cut-off} function $\chi_M$ for $M$: this is a smooth function $\chi_M\in C^\infty_c (M)$ satisfying
\[
\int_G\chi_M (g^{-1}x)dg=1,\quad\mbox{for all}~x\in M.
\]
For any proper action of $G$ on $M$, we consider $ \Omega^\bullet_{{\rm inv}} (M)$, the complex of $G$-invariant differential forms
on $M$ and its cohomology denoted by $H^\bullet_{\rm inv}(M)$. In the universal case this cohomology can be identified with the $K$-relative Lie algebra cohomology of the Lie algebra $\mathfrak{g}$ of $G$: 
$H^\bullet_{\rm inv}(G\slash K)\cong H^\bullet_{CE}(\mathfrak{g};K)$ where $CE$ stands for Chevalley-Eilenberg. For any $\alpha\in\Omega^\bullet_{\rm inv}(G\slash K)$, consider its pull-back $\psi_M^*\alpha\in \Omega^\bullet_{\rm inv}(M)$.
The higher signature associated to $\alpha$ is the real number
\begin{equation}\label{higher-s}
\sigma (M,\alpha) : = \int_M \chi_M L(M)\wedge \psi_M^* (\alpha),
\end{equation}
where $L(M)$ is the invariant de Rham form representing the $L$-class of $M$. 
The insertion of the cut-off function $\chi_M$, which has compact support, ensures that the integral is well-defined, and it can be shown that it only depends on the class $ [L(M)\wedge \psi_M^* (\alpha)] \in H^\bullet_{\rm inv}(M)$.
The collection 
\begin{equation}\label{higher-s-collection}
\{\sigma (M,\alpha), [\alpha] \in H^\bullet_{{\rm inv}} (G/K)\}\end{equation}
are called the higher signatures of $M$.
Similarly, the higher $\widehat{A}$ genus associated to $M$ and to $[\alpha] \in H^\bullet_{{\rm inv}} (G/K)$
is the real number
\begin{equation}\label{higher-a}
\widehat{A}  (M,\alpha) : = \int_M \chi_M \widehat{A}(M)\wedge \psi_M^* (\alpha)\end{equation}
with $\widehat{A}(M)$ the de Rham class associated to the $\widehat{A}$-differential form for a $G$-invariant metric. The collection 
\begin{equation}\label{higher-a-collection}
\{\widehat{A}  (M,\alpha), \alpha \in H^\bullet_{{\rm inv}} (G/K)\}\end{equation}
are called the higher $\widehat{A}$-genera  of $M$.\\

In this paper we establish the following result:

\begin{theorem}\label{theo:main}
Let $G$ be a Lie group with finitely many connected components satisfying property RD, and such that $G/K$
is of non-positive sectional curvature for a maximal compact subgroup $K$.
Let $M$ be an orientable manifold with a proper, cocompact action of $G$. Then the following holds true:
\begin{itemize}
\item[$(i)$] each higher signatures $\sigma (M,\alpha)$, $\alpha \in H^\bullet_{{\rm inv}} (G/K)$,
is a $G$-homotopy invariant of $M$.\\
\item[$(ii)$] if $M$ admits a $G$-invariant spin structure and a $G$-invariant metric of positive scalar curvature,
then each higher $\widehat{A}$-genus $\widehat{A}  (M,\alpha)$, $\alpha \in H^\bullet_{{\rm inv}} (G/K)$,
vanishes.
\end{itemize}
\end{theorem}

\noindent
We prove this result by adapting to the $G$-proper context the seminal paper of Connes and Moscovici
on the cyclic cohomological appraoch to the Novikov conjecture for discrete Gromov hyperbolic groups.\\ 
Crucial to this program is the proof of a higher index formula for higher indices associated to elements in $H^{\bullet}_{{\rm diff}} (G)$ and to the index class $\Ind_{C^*_r (G)} (D)\in K_* (C^*_r (G))$
of a $G$-equivariant Dirac operator on $M$, $M$ even dimensional,  acting on the sections of a complex vector bundle $E$. 
Here are the main steps for establishing this result (for this introduction we expunge from the notation the vector bundle $E$):\\
\begin{enumerate}
\item first, we remark that
for any almost connected Lie group $G$ there is a van Est isomorphism $H^\bullet_{{\rm diff}}(G)\simeq H^\bullet_{{\rm inv}} (G/K)\equiv H^\bullet_{{\rm inv}} (\underline{E}G)$;
\item under the assumption of non-positive sectional curvature for $G/K$ we prove that each $\alpha\in H^{\bullet}_{{\rm diff}}(G)$ 
has a representative cocyle of polynomial growth;
\item if $G$ is unimodular then  for each $\alpha\in H^{{\rm even}}_{{\rm diff}}(G)$ we define a cyclic cocycle $\tau^G_\alpha$ for the convolution algebra
$C^\infty_c (G)$ and thus a homomorphism $\left<\tau^G_\alpha,\cdot\right>: K_0 (C^\infty_c (G))\to \CC$;
\item for each  $\alpha\in H^{{\rm even}}_{{\rm diff}}(G)$ we also consider a cyclic cocycle $\tau^M_\alpha$
for the algebra of $G$-equivariant smooth kernels of $G$-compact support $\mathcal{A}_G^c (M)$;
this defines a homomorphism  $\left<\tau^M_\alpha,\cdot\right>: K_0 (\mathcal{A}_G^c (M))\to \CC$;
\item  we show that if in addition $G$ satisfies the RD property, for example, if $G$ is semisimple connected,
  then $\tau^G_\alpha$ extends to $K_0 (C^*_r (G))$
and  $\tau^M_\alpha$ extends to $K_0 (C^*(M)^G)$, with $C^* (M)^G$ denoting the Roe algebra of $M$;
\item if $D$ is a $G$-equivariant Dirac operator we consider its index class $\Ind_{C^*_r (G)} (D)\in K_0 (C^*_r (G))$
and its Morita equivalent index class $\Ind_{C^*(M)^G} (D)\in K_0 (C^*(M)^G)$ and show that 
 $$\left<\tau^G_\alpha,\Ind_{C^*_r (G)} (D)\right> = \left<\tau^M_\alpha,\Ind_{C^*(M)^G} (D)\right>\,;$$
\item we apply the index theorem of Pflaum-Posthuma-Tang \cite{PPT} in order to compute $\left<\tau^M_\alpha,\Ind_{C^*(M)^G} (D)\right>$, thus establishing our higher $C^*$-index formula in the even dimensional
case.
\end{enumerate}

\medskip
\noindent
We remark that item (2) above has in independent interest, and should be compared with the literature on bounded cohomology of Lie groups, c.f. \cite{ho,kimkim}

The geometric applications stated in Theorem \ref{theo:main} are then a direct consequence of the $G$ homotopy
invariance of the signature index class, established by Fukumoto in \cite{fukumoto2} and, for the higher $\widehat{A}$-genera, of  the
vanishing of the index class   $\Ind_{C^*_r (G)} (\eth)\in K_* (C^*_r (G))$ of  the spin Dirac operator $\eth$
of a $G$-spin $G$-proper manifold endowed with a $G$-metric
of positive scalar curvature, established by Guo, Mathai and Wang in \cite{GMW}. 
In the odd dimensional case we argue by suspension.
Notice that for (certain) 2-degree classes $\alpha$, the $G$-proper homotopy invariance of the higher signatures
$\sigma (M,\alpha)$  had been already established 
by Fukumoto.

\medskip
\noindent
{\bf Acknowledgements.}
Part of this research was carried out during visits of HP to Sapienza Universit\`a di Roma and of
PP to University of Amsterdam. Financial support for these visits was provided by
{\it Istituto Nazionale di Alta Matematica (INDAM)}, through the {\it Gruppo Nazionale per le Strutture Algebriche e Geometriche e loro Applicazioni} (GNSAGA), by the {\it Ministero Istruzione Universit\`a Ricerca (MIUR)},
through the project {\it Spazi di Moduli e Teoria di Lie}, and by NWO TOP  grant nr. 613.001.302.\\
We  thank Andrea Sambusetti, Filippo Cerocchi, Nigel Higson, Varghese Mathai, Xiang Tang and  Hang Wang for many informative and useful discussions.\\

\section{Preliminaries: Proper actions and cohomology}
\subsection{Proper actions}
In this section we introduce the geometric setting for this paper, and list some basic tools that we will need at several points later on.
Let $G$ be a Lie group with finitely many connected components. Recall that a smooth left action of $G$ on a manifold 
$M$ is called {\em proper} if the associated map 
\[
G\times M\to M\times M,\quad (g,x)\mapsto (x,gx),\quad g\in G,x\in M,
\]
is a proper map. This implies that the stabilizer groups $G_x$ of all points $x\in M$ are compact and that the quotient space $M\slash G$ is Hausdorff. The action is 
said to be {\em cocompact} if the quotient is compact. 

The class of manifolds equipped with a proper action of $G$ can be assembled into a category where the morphisms are given by $G$-equivariant smooth 
maps. It is a basic fact that this category has a final object $\underline{E}G$ meaning that any proper $G$-action on $M$ is classified by a $G$-equivariant map $\psi:M\to \underline{E}G$, unique up to $G$-equivariant homotopy. This $\underline{E}G$ is called the {\em classifying space for proper $G$-actions}, and in fact we can take $\underline{E}G:=G\slash K$, where $K$ is a maximal compact subgroup. Then, by writing $S:=\psi^{-1}(eK)$ 
we see that the $S$ is in fact a global slice: it is a $K$-stable submanifold for which there is a diffeomorphism
\[
G\times_K S\cong M,\quad [g,x]\mapsto gx, \quad g\in G, x\in S.
\]
The existence of such a global slice for proper Lie group actions with finitely many connected components was first proved in \cite{Abels}.
When the action is cocompact, $S$ is compact as well.  Closely related to the global slice is the existence of a {\em cut-off} function: this is a smooth function $\chi\in C^\infty(M)$ satisfying
\[
\int_G\chi(g^{-1}x)dg=1,\quad\mbox{for all}~x\in M.
\]
Here we have chosen, for the rest of the paper, a Haar measure which we normalized so that the volume of the maximal compact subgroup $K\subset G$ is equal to $1$. 
When the action of $G$ is cocompact, we can even choose $\chi$ to have compact support. The cut-off function is constructed as follows from the global slice $S\subset M$: choose a 
smooth function $h\in C^\infty(M)$ which is equal to $1$ on $S$ and zero outside an open neighborhood of $S$ in $M$. Then the function
\[
\chi(x)=\left(\int_G h(g^{-1}x)dg\right)^{-1}h(x),
\]
is a cut-off function for the action of $G$. Choosing a $G$-invariant riemannian metric $g$ on $M$ we can refine this construction as follows: choose the initial function $h$ to have 
support inside the tube of distance $1$ in $M$ around $S$. Then, rescaling by $\epsilon>0$ along the radial coordinate near $S$, we obtain a family of functions $h_\epsilon$ 
satisfying
\[
h_\epsilon(x)=\begin{cases} 1& x\in S\\0&d(x,S)>\epsilon.\end{cases}
\]
Using this as input for the construction of the cut-off function above gives a family of cut-off functions $\chi_\epsilon$ satisfying:
\begin{lemma}
\label{cut-off:family}
The family of cut-off functions $\chi_\epsilon,~\epsilon>0$ satisfies
\[
\lim_{\epsilon\downarrow 0}\chi_\epsilon =\chi_S,
\]
distributionally.
\end{lemma}
\begin{proof}
We begin by remarking that pointwise
\[
\lim_{\epsilon\downarrow 0}\chi_\epsilon(x)=\begin{cases} 1&\mbox{for}~x\in S\\ 0&\mbox{for} ~x\not\in S.\end{cases}
\]
This is because for fixed $x\in S$ the family $h_\epsilon(g^{-1}x)$ of functions on $G$ converges pointwise to the characteristic function of $K\subset G$ and therefore by dominated convergence we have
\[
\lim_{\epsilon\downarrow 0}\int_Gh_\epsilon(g^{-1}x)dg=\int_G\lim_{\epsilon\downarrow 0}h_\epsilon(g^{-1}x)dg=\int_K dg=1,
\]
by our normalization of the Haar measure on $G$. With this pointwise limit of $\chi_\epsilon(x)$ we have, once again by dominated convergence that
\[
\lim_{\epsilon\downarrow 0}\int_M\chi_\epsilon(x)f(x) dx=\int_M\lim_{\epsilon\downarrow 0}\chi_\epsilon(x)f(x) dx=\int_Sf(x)dx,
\]
for any test function $f\in C^\infty_c(M)$.
\end{proof}

\subsection{Invariant cohomology and the van Est map} 
The main point of this subsection is to define the van Est map associated to a proper action of a Lie group $G$ on $M$, and to reinterpret this map as the pull-back in cohomology along the classifying 
map $\psi_M:M\to G\slash K$.

Let $M$ be a smooth manifold equipped with a smooth proper action of $G$. We define
\[
\Omega^\bullet_{\rm inv}(M):=\{\omega\in\Omega^\bullet(M),~g^*\omega=\omega,~\forall g\in G\},
\]
the vector space of invariant differential forms. The de Rham differential restricts to this space of invariant forms and its cohomology, called the {\em invariant cohomology},
is denoted by $H^\bullet_{\rm inv}(M)$. Taking the invariant cohomology defines a contravariant functor on the category of proper $G$-manifolds with an equivariant map $f:M\to N$ 
acting on cohomology by pull-back of differential forms as usual. It is not difficult to see that the induced map $f^*:H^\bullet_{\rm inv}(N)\to H^\bullet_{\rm inv}(M)$ depends only 
on the $G$-homotopy class it is in. Given the choice of a cut-off function $\chi$, it is shown in \cite{PPT} that the integral 
\[
\int_M\chi\alpha
\]
of a closed form $\alpha\in\Omega^{\dim(M)}_{\rm inv, cl}(M)$, only depend on the cohomology class $[\alpha]\in H^{\dim(M)}_{\rm inv}(M)$.

For any manifold $M$ equipped with a proper action of $G$, the {\em van Est map} is a map $H^\bullet_{\rm diff}(G)\to H^\bullet_{\rm inv}(M)$, where $H^\bullet_{\rm diff}(G)$ is the 
so-called {\em smooth group cohomology} of $G$. Let us first recall the definition of this smooth group cohomology. For $G$ a Lie group, 
the space of smooth homogeneous group $k$-cochains  is given by 
\[
C^k_{\rm diff}(G):=\{c:G^{\times(k+1)}\to\mathbb{C}~{\rm smooth},~c(gg_0,\ldots,gg_k)=c(g_0,\ldots,g_k),~{\rm for all}~g,g_0,\ldots,g_k\in G\}.
\]
The differential $\delta:C^k_{\rm diff}(G)\to C^{k+1}_{\rm diff}(G)$  is defined as
\begin{equation}
\label{diff-gc}
(\delta c)(g_0,\ldots,g_{k+1}):=\sum_{i=0}^{k+1}(-1)^ic(g_1,\ldots,\hat{g}_i,\ldots,g_{k+1}),
\end{equation}
where the $\hat{}$ means omission from the argument of the function.
The cohomology of the resulting complex is called the smooth group cohomology, written as $H^\bullet_{\rm diff}(G)$. 

With this, the van Est map is constructed as follows: given a smooth group $c\in C^k_{\rm diff}(G)$,
define the differential form
\begin{equation}
\label{dfg}
\omega^\chi_c:=(d_1\cdots d_k f_c)|_\Delta
\end{equation}
where $d_i$ means taking the differential in the i'th variable of the function $f_c\in C^\infty(M^{\times(k+1)})$ defined as 
\begin{equation}
\label{mc}
f_c(x_0,\ldots,x_k):=\int_{G^{\times(k+1)}}\chi(g_0^{-1}x_0)\cdots\chi(g_k^{-1}x_k)c(g_0,\ldots,g_k)d\mu(g_0)\cdots d\mu(g_k).
\end{equation}
\begin{proposition}
The map $c\mapsto \omega^\chi_c$ defines a morphism of complexes 
\[
\Phi^\chi_M:\left(C^\bullet_{\rm diff}(G),\delta\right)\longrightarrow \left(\Omega^\bullet_{\rm inv}(M),d_{dR}\right).
\]
On the level of cohomology, it is independent of the choice of cut-off function $\chi$.
\end{proposition}
\begin{remark}
Because of this last property, we will often omit the superscript $\chi$ and write $\omega_c$ and $\Phi_M$ when the context only refers to the cohomological 
meaning of the differential form and the van Est map.
\end{remark}
\begin{proof}
We start by giving the abstract cohomological definition of the map $\Phi_M$ following \cite{crainic} using a double complex, after which we show how to obtain the explicit chain 
morphism by constructing a splitting of the rows. The double complex is given as follows. We define
\[
C^{p,q}:=C^\infty(G^{\times (p+1)},\Omega_{\rm inv}^q(M)).
\]
The vertical differential $\delta_v:C^{p,q}\to C^{p,q+1}$ is simply given by the de Rham differential, leaving the $G$-variables untouched. As for the horizontal differential 
$\delta_h:C^{p,q}\to C^{p+1,q}$: this is given by differential computing the smooth groupoid cohomology of the action groupoid $G\times M\rightrightarrows M$ with coefficients in
$\bigwedge^qT^*M$, viewed as a representation of this groupoid. Since the $G$-action is proper, the groupoid  $G\times M\rightrightarrows M$ is proper by definition. Therefore,
the vanishing theorem for the groupoid cohomology of proper Lie groupoids in \cite{crainic} applies, and we see that the rows in this double complex are exact. 
There are obvious inclusions 
$C^\bullet_{\rm diff}(G)\hookrightarrow C^{\bullet,0}$, and $\Omega_{\rm inv}^\bullet(M)\hookrightarrow C^{0,\bullet}$, and now we see that by finding the appropriate splittings we can
 "zig-zag" from one end to the other in the double complex:
\[
\xymatrix{\ldots&\ldots&\ldots&\ldots&\\ \Omega^1(M)\ar[u]^{d}\ar[r]&C^{0,1}\ar[u]^{\delta_v}\ar[r]^{\delta_h}&C^{1,1}\ar[u]^{\delta_h}\ar@/^1pc/@{-->}[l]^{s_{1}}\ar[r]^{\delta_h}&C^{2,1}\ar[u]^{\delta_h}\ar[r]^{\delta_h}&\ldots\\
\Omega^0(M)\ar[u]^{d}\ar[r]&C^{0,0}\ar[u]^{\delta_v}\ar[r]^{\delta_h}&C^{1,0}\ar[u]^{\delta_h}\ar[r]^{\delta_h}&C^{2,0}\ar[u]^{\delta_h}\ar@/^1pc/@{-->}[l]^{s_{2}}\ar[r]^{\delta_h}&\ldots\\
&C^0_{\rm diff}(G)\ar[u]\ar[r]^\delta &C^1_{\rm diff}(G)\ar[u]\ar[r]^\delta&C^2_{\rm diff}(G)\ar[u]\ar[r]^{\delta}&\ldots
}
\]
So it remains to find an appropriate splitting $s_p:C^{p,\bullet}\to C^{p+1,\bullet}$. Given a choice of cut-off function $\chi$, the formula
\[
(s_p\alpha)(g_0,\ldots,g_{p-1}):=\left.d_{x_0}\int_G\chi(g^{-1}x_0)\alpha(g,g_0,\ldots,g_{p-1})\right|_\Delta,\quad\alpha\in C^{p,q},
\]
does the job:  a straightforward computation shows that
\[
\delta_h\circ s+s\circ\delta_h={\rm id}.
\]
With this choice of contraction map, one obtains exactly equation \eqref{dfg} for the invariant differential form associated to a group cochain. The preceeding argument therefore shows that the map $c\mapsto \omega_c$ is indeed 
a morphism of cochain complexes.
\end{proof}
\begin{remark}[The van Est isomorphism] 
\label{van Est}
The main theorem of \cite{crainic} states that if $M$ is cohomologically $n$-connected, the map $\Phi_M$ induces an isomorphism in cohomology in degree $\leq n$ and is injective in degree $n+1$.
In the universal case for the action of $G$ on $G\slash K$, which is contractible, we therefore find an isomorphism $H^\bullet_{\rm diff}(G)\cong H^\bullet_{\rm inv}(G\slash K)$. This is one version of the classical van Est theorem \cite{van-Est}. 
In this case we have, by left translation
\begin{equation}
\label{idl}
\Omega^\bullet_{\rm inv}(G\slash K)\cong\left(\bigwedge^\bullet(\mathfrak{g}\slash\mathfrak{k})^*\right)^K,
\end{equation}
under which the de Rham differential identifies with the Chevalley--Eilenberg differential computing the relative Lie algebra cohomology $H^\bullet_{CE}(\mathfrak{g};K)$. With this, the van Est isomorphism
is written as
\begin{equation}
\label{vanEst}
H^\bullet_{\rm diff}(G)\cong H^\bullet_{CE}(\mathfrak{g};K).
\end{equation}
\end{remark}
\begin{proposition}
Let $f:M\to N$ be an equivariant smooth map between proper $G$-manifolds. Then the following diagram commutes:
\[
\xymatrix{H^\bullet_{\rm diff}(G)\ar[rr]^{\Phi_N}\ar[rrd]_{\Phi_M}&&H^\bullet_{\rm inv}(N)\ar[d]^{f^*}\\&&H^\bullet_{\rm inv}(M)}
\]
\end{proposition}
\begin{proof}
Let $\chi_M$ be a cutt-off function for the $G$-action on $M$. Then the pull-back $f^*\chi_M$ is a cut-off function for the $G$-action on $N$. For this 
cut-off function we obviously have $\omega^{f^*\chi_M}_c=f^*\omega_c^{\chi_M}$. The result now follows from the fact that the van Est map is independent of the 
choice of cut-off function.
\end{proof}

\begin{corollary}
Under the van Est isomorphism $H^\bullet_{\rm diff}(G)\cong H^\bullet_{\rm inv}(G\slash K)$, the van Est map is identified with the 
pull-back along the classifying map $\psi_M:M\to G\slash K$, i.e., $\Phi_M=\psi_M^*$.
\end{corollary}

\subsection{Group cocycles of polynomial growth}
\label{egc}
In a later stage of the paper, in the discussion of the extension properties of cyclic cocycles associated to smooth group cocycles, it will be  important to control the growth
of the group cocycles. To this end, we shall prove below a criterium guaranteeing that we can represent classes in $H^\bullet_{\rm diff}(G)$ by cocycles that have at most polynomial growth. For this, we begin by recalling Dupont's inverse
\cite{Dupont-Simplicial} of the van Est map $\Phi_{G\slash K}$ establishing the isomorphism \eqref{vanEst}.
Choose an ${\rm Ad}_K$-invariant inner product $\left<~,~\right>$ on $\mathfrak{g}$, which, by left translations, induces a $G$-invariant Riemannian metric on $G\slash K$.
This metric defines an orthogonal decomposition $\mathfrak{g}=\mathfrak{p}\oplus\mathfrak{k}$ with $\mathfrak{p}\cong T_{eK}(G\slash K)$. Since $K$ is maximal compact, 
the (riemannian) exponential map induces an isomorphism $\mathfrak{p}\cong G\slash K$ (with inverse denoted by $\log$), and we can define the contraction
\[
\varphi_s(x):=\exp(s\log(x)),
\]
of $G\slash K$ to its basepoint $eK\in G\slash K$, i.e., $\varphi_1={\rm id}_{G\slash K}$, and $\varphi_0(x)=eK$. 
Now, given $k+1$ points $g_0K,\ldots,g_kK\in G\slash K$, also denoted 
$\bar{g}_0,\ldots,\bar{g}_k$,
we can  consider the geodesic simplex
$\Delta^k(g_0K,\ldots,g_kK)\subset G\slash K$ defined inductively as the cone of $\Delta^{k-1}(\bar{g}_1,\ldots,\bar{g}_k)$ with tip point $\bar{g}_0$.
More precisely,  define the singular simplex $\sigma^k(\bar{g}_0,\ldots,\bar{g}_k):\Delta^k\to G\slash K$, where $\Delta^k:=\{(t_0,\ldots,t_k)\in\mathbb{R}^{k+1},~t_i\geq 0,\sum_it_i=1\}$, by 
\[
\sigma^k(g_0K,\ldots,g_kK)(t_0,\ldots,t_k):=g_0\varphi_{t_0}\left(\sigma^{k-1}(g_0^{-1}g_1K,\ldots,g_0^{-1}g_k K)\left(\frac{t_1}{1-t_0},\ldots, \frac{t_k}{1-t_0}\right)\right),
\]
and $\sigma^0(gK):=gK$. We write $\Delta^k(g_0K,\ldots,g_kK)$ for the image of this simplex. 
By construction, this $k$-simplex is $G$-invariant: $g\Delta^k(\bar{g}_0,\ldots,\bar{g}_k)=\Delta^k(g\bar{g}_0,\ldots,g\bar{g}_k)$. With these simplices we define a map
\begin{equation}
\label{cadf}
J:\Omega^\bullet_{\rm inv}(G\slash K)\longrightarrow C^\bullet_{\rm diff}(G),\quad \alpha\mapsto J(\alpha)(g_0,\ldots,g_k):= \int_{\Delta^k(g_0K,\ldots,g_kK)}\alpha,
\end{equation}
which is easily checked to be a morphism of cochain complexes. Since $\Phi_{G\slash K}\circ J={\rm id}$, $J$ is a quasi-isomorphism.

\begin{theorem}\label{cocycle:pol-est}
Let $G$ be a Lie group with finitely many connected components.
Let $K$ be a maximal
compact subgroup and assume that $G/K$  is of non-positive sectional curvature
with respect to the G-invariant  metric induced 
by a  AdK-invariant inner product $\left<~,~\right>$ on $\mathfrak{g}$.
Then the group cocycle associated to a closed $\alpha\in \Omega^k_{\rm inv}(G\slash K)$ has polynomial growth. More precisely, if we write $d(g)$ for the distance from $eK$ to $gK$ in $G\slash K$, there exists a constant $C>0$ such that the following 
estimate holds true:
\[
|J(\alpha)(g_0,\ldots,g_k)|\leq C(1+d(g_0))^{k}\cdots (1+d(g_k))^{k}.
\]
\end{theorem}
\begin{proof}
We denote by $||\alpha||$ the norm of the Lie algebra cocycle 
$\alpha\in C^k_{\rm CE}(\mathfrak{g};K)=\Omega^k_{\rm inv}(G\slash K)$ defined by the $K$-invariant metric on the Lie algebra $\mathfrak{g}$ of $G$ that defines the metric on 
$G\slash K$. Since $\alpha$ is a $G$-invariant differential form we obviously have the inequality
\begin{align*}
|J(\alpha)(g_0,\ldots,g_k)|\leq ||\alpha||{\rm Vol}(\Delta^k(\bar{g}_0,\ldots,\bar{g}_k)).
\end{align*}
We will now prove that, under the assumptions of the Lemma,  the volume of the geodesic $k$-simplex on $G\slash K$ has at most polynomial growth in the geodesic distance of its vertices, thus completing the proof the Lemma.
For this we adapt an argument  from \cite{HK}; we thank Andrea Sambusetti for bringing
this article to our attention.

First remark that as $\varphi_s(gK)$ is the geodesic connecting 
$gK$ to the base point $eK$, the simplex $\Delta^k(g_0K,\ldots,g_kK)$  has, by construction, the property that for any interior point $x\in \Delta^k(g_0K,\ldots,g_kK)$ there are $k$ geodesics $\gamma_1(s_1),\ldots,\gamma_k(s_k)$, 
each connecting two points in the boundary $\partial \Delta^k(g_0K,\ldots,g_kK)$ passing through $x$ whose velocities  span $T_x\Delta^k(g_0K,\ldots,g_kK)$.
Consider the one-parameter family of simplices 
\[
s\mapsto \Delta^k(eK,\varphi_s(g_1K)\ldots\varphi_s(g_kK)),~s\in [0,1]
\] 
contracting $\Delta^k(eK,g_1K,\ldots,g_kK)$ to the basepoint $eK$.
This contraction is generated by a vector field $Y$ which has the property that it is a Jacobi field with respect to each geodesics $\gamma_i(s_i)$ mentioned above passing through 
the point $x\in \Delta^k(eK,\varphi_s(g_1K),\ldots\varphi_s(g_kK))$. The Jacobi equation satisfied by $Y$ therefore gives
\begin{align*}
\frac{d^2}{ds^2_i}||Y(s_i)||^2&=2||\nabla_{\partial/\partial s_i} Y||^2-2\left<R\left(Y(s_i),\frac{d\gamma_i(s_i)}{ds_i}\right)\frac{d\gamma_i(s_i)}{ds_i},Y(s_i)\right>\\
&\geq 0,
\end{align*}
since the sectional curvature of $G/K$ is non-positive. The maximum principle gives that $||Y(s_i)||$ attains its maximum at one of the points $s_i=0,1$. Since this holds true for any $i$, we conclude that the 
maximum is attained on $\partial\Delta^k(eK,\ldots,\varphi_s(g_kK))$, and, proceeding inductively on $k$, in one of the vertices $\varphi_s(g_iK),~i=1,\ldots, k$. But on these vertices, 
$s\mapsto \varphi_s(g_iK)$ is simply the geodesic connecting $eK$ with $g_iK$, which is generated by  the Euler vector field $\sum_i X^i\frac{\partial}{\partial X^i}$ on $\mathfrak{p}$ which has polynomial growth of degree $1$ in the geodesic distance. 
Since the exponential map $\exp:\mathfrak{p}\to G\slash K$ is a radial isometry, the same holds true on $G\slash K$, and we can conclude that the vector field $Y$ has polynomial growth of degree $1$.

Remark that the generating vector field $Y$ is tangent to $\Delta^k(eK,g_1K,\ldots,g_kK)$ to all the boundary faces except for $\Delta^{k-1}(g_1K,\ldots,g_kK)\subset \partial \Delta^k(eK,g_1K,\ldots,g_kK))$. The standard variational formula for the volume therefore gives 
\begin{align*}
\frac{d}{ds}\left({\rm Vol}(\Delta^k(eK,\varphi_s(g_1K),\ldots,\varphi_s(g_kK)))\right)&=\int_{\Delta^k(eK,\varphi_s(g_1K),\ldots,\varphi_s(g_kK))}{\rm div}(Y)\, d{\rm vol}_{\Delta^k}\\
&=\int_{\partial \Delta^k(eK,\varphi_s(g_1K),\ldots,\varphi_s(g_kK))}(Y\cdot n) \, d{\rm vol}_{\partial\Delta^k}\\
&=\int_{\Delta^{k-1}(\varphi_s(g_1K),\ldots,\varphi_s(g_kK))}(Y\cdot n)\, d{\rm vol}_{\Delta^{k-1}}\\
&\leq \prod_i (1+d(g_i)){\rm Vol}(\Delta^{k-1}(\varphi_s(g_1K),\ldots,\varphi_s(g_kK)),
\end{align*}
where $n$ denotes the vector field normal to the boundary, and we have used the fact that $Y\cdot n$ is zero on all faces except $\Delta^{k-1}(g_1K,\ldots,g_kK)\subset \partial \Delta^k(eK,g_1K,\ldots,g_kK))$.
We now use induction: for $k=1$, the geodesic simplex $\Delta^1(g_0K,g_1K)$ is simply the geodesic line segment connecting $g_0K$ and $g_1K$, so the estimate
in the Lemma obviously holds true. Assume now that the estimate holds true for all degrees up to $k-1$. Then, by the mean value theorem:
\begin{align*}
{\rm Vol}(\Delta^k(eK,g_1K,\ldots,g_kK))&=\left.\frac{d}{ds}\left({\rm Vol}(\varphi_s(\Delta^k(eK,g_1K,\ldots,g_kK))\right)\right|_{s=s_0},\quad\mbox{for some}~s_0\in[0,1],\\
&\leq C\prod_{i=1}^k(1+d(g_i)).
\end{align*}
The estimate for the general simplex $\Delta^k(g_0K,\ldots,g_kK)$ follows by translation over $g_0$ (which is an isometry) together with the triangle inequality. 
\end{proof}

\begin{example}
\label{ex-gc}
As an example, consider the abelian group $G=\mathbb{R}^2$ with maximal compact group given by the trivial group $\{0\}\subset \mathbb{R}^2$. In this abelian case we have that $H^\bullet_{\rm inv}(\mathbb{R}^2)=\bigwedge^\bullet\mathbb{R}^2$, 
and a generator in degree $2$ is given by the area form $dx\wedge dy$, so that we find 
\begin{equation}
\label{area-cocycle}
J(dx\wedge dy)(x,y,z)={\rm Area}_{\RR^2}(\Delta^2(x,y,z)), 
\end{equation}
which evidently grows polynomially in the norm of $x,y$ and $z$. 
\end{example}

\begin{remark}{\ }
\label{ex-gc-semisimple}
\begin{itemize}
\item[$i)$]
When $G$ is a connected semisimple Lie group, $G\slash K$ is a non-compact symmetric space and has non-positive sectional curvature \cite{Helgason}. Therefore the curvature assumptions in the Lemma
are automatically satisfied in this case. 
In fact, the conjecture in \cite{Dupont} is that for semisimple Lie groups all these cocycles are bounded. 
For recent work on this conjecture, see \cite{ho,kimkim}. In this last reference, different simplices are used, given by the barycentric subdivision of the geodesic ones,
to prove boundedness of the top dimensional cocycle for general connected semisimple Lie groups.
\item[$ii)$]
In general, the polynomial bounds of the Lemma above are not sharp, as expected from the conjecture mentioned in $i)$. 
For example, when $G=SL(2,\mathbb{R})$, the maximal compact subgroup is given by $K=SO(2)$ so that 
$G\slash K=\mathbb{H}^2$, the hyperbolic $2$-plane. Again, we have $H^2_{\rm inv}(\mathbb{H}^2)=\mathbb{R}$, 
with generator the hyperbolic area form. This leads to a smooth group cocycle given by the same formula as 
\eqref{area-cocycle} above, replacing the Euclidean area by the hyperbolic one, but this time the cocycle is bounded because the area of a hyperbolic triangle does not 
exceed $\pi$, confirming the boundedness in top-degree mentioned in $i)$.
\end{itemize}
\end{remark}

\section{Algebras of invariant kernels}
\label{section:algebras-closed}

\subsection{Smoothing kernels of $G$-compact support.}
Let $M$ as above,  a closed smooth manifold carrying a smooth proper action of a Lie group $G$ with $|\pi_0 (G)|<\infty$
and with compact quotient. 
We choose an invariant complete Riemannian metric, denoted $h$, with associated distance function denoted by $d_{M}(x,y)$ for $x,y\in M$, and volume form $d{\rm vol}(x)$. We fix a left-invariant metric on $G$ and we denote by 
 $d_{G}$ the associated distance function.
\begin{definition}
\label{def-alg-kern}
Consider  a $G$-equivariant smoothing kernel $k\in C^{\infty}(M\times M)$; thus $k$ is an element
in $C^{\infty}(M\times M)^{G\times G}$. We  say that
 $k$ is of $G$-compact support  if the projection of $\supp(k)\subset M\times M$ in $(M\times M)/G$, with $G$ acting diagonally, is compact. 
 \end{definition}

 \noindent
We denote by $\mathcal{A}_{G}^{c}(M)$ the set of $G$-equivariant smoothing kernels of $G$-compact support.
It is well known that $\mathcal{A}_{G}^{c}(M)$ has the structure of a Fr\'echet  algebra with respect to 
the convolution product
$$(k * k^\prime) (x,z)=\int_M k(x,y) k^\prime (y,z) d{\rm vol}(y)$$
It is also well known that each element $k\in \mathcal{A}_{G}^{c}(M)$ defines an equivariant linear operator $S_k :C^\infty_c (M)\to C^\infty_c (M)$, the integral operator associated to the kernel $k$,
and  that $S_{k}\circ S_{k^\prime}=S_{k * k^\prime}$. Moreover, $S_k$ extends to an equivariant bounded operator 
on $L^2 (M)$. 
We have therefore defined a subalgebra of $\mathcal{B}(L^2 (M))$ that we denote as $\mathcal{S}_G^c (M)$; by definition
\begin{equation}\label{sGc}
\mathcal{S}_G^c (M):=\{S_k, k\in  \mathcal{A}_{G}^{c}(M)\}.
\end{equation}

The case in which there is an equivariant vector bundle
$E$ on $M$ is similar, in that we start with $G$-equivariant elements in $C^{\infty}(M\times M, E\boxtimes E^*)$ and then proceed analogously,
defining in this way the Fr\'echet algebra $\mathcal{A}_{G}^{c}(M,E)$ and $\mathcal{S}_G^c (M,E):=
\{S_k, k\in  \mathcal{A}_{G}^{c}(M,E)\}$, a subalgebra of $\mathcal{B}(L^2 (M,E))$.

\medskip
\noindent
{\bf Notation.} Keeping with a well establised abuse of notation, we shall often identify
$\mathcal{A}_{G}^{c}(M,E)$ with $\mathcal{S}_{G}^{c}(M,E)$
thus identifying a smoothing kernel $k$ in $\mathcal{A}_{G}^{c}(M,E)$ with the corresponding
operator $S_k\in \mathcal{S}_{G}^{c}(M,E)$.

\subsection{Holomorphically closed subalgebras}\label{subsect:holo}
Using the remarks at the end of the previous subsection we see that  that $\mathcal{S}_G^c (M,E)$ is in an obvious way a subalgebra of the 
reduced Roe $C^*$-algebra $C^*(M,E)^G$. Recall that $C^*(M,E)^G$ is defined as a the norm closure in $\mathcal{B}(L^2 (M,E))$
of the algebra $C^*_c (M,E)^G$  of $G$-equivariant bounded operators of finite propagation and locally compact. 
In fact, $\mathcal{S}_G^c (M,E)\subset C^*_c (M,E)^G$. The Roe algebra is canonically  isomorphic to $\mathbb{K}(\mathcal{E})$, the $C^*$-algebra of
compact operators
of the Hilbert $C^*_r (G)$-Hilbert module $\mathcal{E}$ obtained by closing the space of compactly supported 
sections of $E$ on $M$, $C^\infty_c (M,E)$, endowed with the  $C^*_r G$-valued inner product
\begin{equation}\label{C*-inner-product}
 (e,e^\prime)_{C^*_r G} (x):= (e,x\cdot e^\prime)_{L^2 (M,E)}, \quad e,e^\prime \in C^\infty_c (M,E)\,,\;x\in G\,.
 \end{equation} See for example \cite{Hochs-Wang-HC} where
the Morita isomorphism
$$K_* (\mathbb{K}(\mathcal{E}))=K_* (C^*(M,E)^G)\xrightarrow{\mathcal{M}} K_* (C^*_r G)$$
is explicitly discussed. We shall come back to this important point in a moment.\\ 
The subalgebra  $\mathcal{S}_G^c (M,E)$ 
is not holomorphically closed in $C^*(M,E)^G$. On the other hand, such a   subalgebra  of $C^*(M,E)^G$
is implicitly constructed in \cite[Section 3.1]{Hochs-Wang-HC}  by making use of the slice theorem. We recall the essential ingredients, following \cite[Section 3.1]{Hochs-Wang-HC} (we also extend the context slightly for future use).

As already remarked in the previous section,  under our assumptions on $G$, there exists a global slice
for the action of $G$ on $M$: thus if  $K$
is a maximal compact subgroup of $G$ there exists a K-invariant compact submanifold $S\subset M$ such that the action map $[g,s]\to gs$, $g\in G$, $s\in S$, defines a $G$-equivariant
diffeomorphism 
$$ G\times_K S \xrightarrow{\alpha} M$$
where $S$ is  compact because the action is cocompact. Corresponding to this diffeomorphism we have an isomorphism $E\cong G\times_K (E|_S) $ and thus 
isomorphisms
$$C^\infty_c (M,E)\cong (C^\infty_c (G)\hat{\otimes}C^\infty (S,E|_{S}))^K\,,\qquad C^\infty (M,E)\cong (C^\infty (G)\hat{\otimes} C^\infty (S,E|_{S}))^K\,.$$ See \cite[Section 3.1]{Hochs-Wang-HC}.
Here we are taking the projective tensor product $\hat{\otimes}_\pi$ of the two Fr\'echet algebras; however, since 
$ C^\infty (S,E|_{S})$ is nuclear, the injective $\hat{\otimes}_{\epsilon}$ and projective $\hat{\otimes}_{\pi}$
tensor products coincide, which is why we do not use a subscript.
Consider now $\Psi^{-\infty}(S, E|_S)$, also a nuclear Fr\'echet algebra, and let 
$$\widetilde{A}_G^c (M,E):= (C^\infty_c (G)\hat{\otimes} \Psi^{-\infty}(S,E|_S))^{K\times K}$$
$\widetilde{A}_G^c (M,E)$ is a Fr\'echet algebra, with product denoted by $*$.
Let $\widetilde{k}\in \widetilde{A}_G^c (M,E)$ and consider the operator $T_{\widetilde{k}}$
on $L^2 (M,E)$ given by
\begin{equation}\label{correspondence-slice}
(T_{\widetilde{k}} e)(gs)=\int_G \int_S g \widetilde{k}(g^{-1}g^\prime,s,s^\prime) g^{\prime\,-1} e (g^\prime s^\prime)ds^\prime dg^\prime\end{equation}
This is a bounded $G$-equivariant operator with smooth $G$-equivariant Schwartz kernel given by
$$\kappa (gs,g^\prime s^\prime)=g \widetilde{k} (g^{-1} g^\prime, s, s^\prime) g^{\prime\,-1}$$
where the $g$ and $g^{\prime,-1}$ on the right hand side are used in order to 
identify fibers on the vector bundle $E$. 
The assignment 
$\widetilde{k}\to  T_{\widetilde{k}} $ is injective and satisfies 
$$T_{\widetilde{k}} \circ T_{\widetilde{k}^\prime} = T_{\widetilde{k}* \widetilde{k}^\prime } \,.$$
Consider  the subalgebra of the bounded operators on  $L^2 (M,E)$
given by
\begin{equation*}
\{T_{\widetilde{k}}, \;\widetilde{k}\in \widetilde{A}_G^c (M,E)\}\,
\end{equation*}
endowed with the Fr\'echet algebra structure induced by the injective homomorphism $\tilde{k}\to
T_{\tilde{k}}$.
It is easy to see that this algebra is precisely equal to the algebra we have considered in
the previous subsection, $\mathcal{S}_G^c (M,E):=\{S_k, k\in \mathcal{A}^c_G (M,E)\}$. Thus, in formulae,
\begin{equation}\label{equalc}
\mathcal{S}_G^c (M,E)=\{T_{\widetilde{k}}, \;\widetilde{k}\in \widetilde{A}_G^c (M,E)\}
\,.
\end{equation}
 Summarizing:
using the slice theorem we have realized $\mathcal{S}^c_G (M,E)$ as a projective tensor product
of convolution operators on $G$ and smoothing operators on $S$. This preliminary result puts us in the position of 
enlarging the algebra $\mathcal{S}^c_G (M,E)$ and obtain a subalgebra dense and holomorphically
closed in $C^* (M,E)^G$.\\ 
 To this end we give the following definition.
\begin{definition}\label{def-ag}
Let $\mathcal{A}(G)$ a set of functions on $G$. We shall say that $\mathcal{A}(G)$
is admissible if the 
 following properties are satisfied:
\begin{enumerate}
\item   $\mathcal{A}(G)$ is a Fr\'echet space  and there are continuous inclusions
$ C^\infty_c (G)\subset \mathcal{A}(G)\subset L^2 (G)$;
\item the action by convolution defines  a continuous injective map
$\mathcal{A}(G) \hookrightarrow  C^*_r (G)$ which makes  $\mathcal{A}(G)$ a  subalgebra of $C^*_r (G)$;
\item  $\mathcal{A}(G)$ is holomorphically closed
in $C^*_r (G)$
\end{enumerate}
\end{definition}

We can then consider
$$\widetilde{A}_G (M,E):= (\mathcal{A}(G)\hat{\otimes} \Psi^{-\infty}(S,E|_S))^{K\times K}$$
a Fr\'echet algebra and for  $\widetilde{k}\in \widetilde{A}_G (M,E)$ the bounded operator $T_{\widetilde{k}}$
on $L^2 (M,E)$ given by
\begin{equation}\label{correspondence-slice}
(T_{\widetilde{k}} e)(gs)=\int_G \int_S g \widetilde{k}(g^{-1}g^\prime,s,s^\prime) g^{\prime\,-1} e (g^\prime s^\prime)ds^\prime dg^\prime\end{equation}
The operator $T_{\widetilde{k}}$ is an integral operator with  $G$-equivariant Schwartz kernel
$\kappa$ given by $\kappa (gs,g^\prime s^\prime)=g \widetilde{k} (g^{-1} g^\prime, s, s^\prime) g^{\prime\,-1}\,.$
Since $\mathcal{A}(G) \hookrightarrow  C^*_r (G)$, with  $\mathcal{A}(G)$ acting by convultion,
we see that $T_{\widetilde{k}}$ is $L^2$-bounded.

\begin{definition}
We define $\mathcal{A}_G (M,E)$ as the subalgebra of the bounded operators on  $L^2 (M,E)$
given by
$$\mathcal{A}_G (M,E) :=\{T_{\widetilde{k}}, \widetilde{k}\in \widetilde{A}_G(M,E)\}\,.$$
We endow  $\mathcal{A}_G (M,E) $ with the structure of Fr\'echet algebra  induced by the injective 
homomorphism $\tilde{k}\to T_{\tilde{k}}$.
\end{definition}

\begin{proposition}
\label{lemma:HC-holoclosed}
Under the assumptions (1)--(3) for $\mathcal{A}(G)$ appearing in Definition
\ref{def-ag}  the following holds:
\begin{itemize}
\item[(i)] We have a continuous  inclusion of Fr\'echet algebras 
\begin{equation}\label{inclusions}
\mathcal{S}^c_G (M,E) \subset    \mathcal{A}_G(M,E) 
\end{equation}
\item[(ii)]   $\mathcal{A}_G(M,E)$  is a dense subalgebra of $C^*(M,E)^G$ and it is   holomorphically closed.
\end{itemize}
\end{proposition}
\begin{proof}
(i) The continuous inclusion of Fr\'echet algebras  $\mathcal{S}^c_G (M,E)  \subset   \mathcal{A}_G(M,E)$ follows immediately
from \eqref{equalc}.\\
(ii) The fact that  $\mathcal{A}_G(M,E)$  is a dense subalgebra of $C^*(M,E)^G$ is proved
precisely as  in \cite[Lemma 3.6]{Hochs-Wang-HC}; the property of being holomorphically closed
 follows easily from the
hypothesis that 
 $\mathcal{A}(G)$
is holomorphically closed in $C^*_r G$ and the well known fact that $\Psi^{-\infty}(S,E|_S)$ is holomorphically closed 
in the compact operators of $L^2 (S,E|_S)$.
\end{proof}

\begin{definition}
\label{prop-rd}
Let $G$ be a Lie group and let $L$ be a length function on $G$. We consider 
\begin{equation}\label{RD}
H^\infty_L (G)=\{f\in L^2 (G)\;:\; \int_G (1+L(x))^{2k} |f(x)|^2 dx <+\infty \;\;\forall k\in\mathbb{N}\}
\end{equation}
endowed with the Fr\'echet topology induced by the sequence of seminorms
\begin{equation}
\label{seminorm}
\nu_k(f):=||(1+L)^kf||_{L^2}.
\end{equation}
We shall say that the pair $(G,L)$ satisfies the Rapid Decay property (RD) if there is a continuous
inclusion $H^\infty_L (G)\hookrightarrow   C^*_r (G)$.
\end{definition}

For conditions equivalent to the one given here, see  \cite{Chatterji-et-al-duke}. 
We also recall that if $G$ satisfies (RD) then $G$ is unimodular. See \cite{ji-s-K}.

\begin{proposition}\label{prop:semis}
Let $G$ be a Lie group with $|\pi_0 (G)|<\infty$; we can and we shall choose 
$L$ to be  the length function
associated to a left-invariant Riemannian metric. Assume additionally that $G$ satisfies
(RD) (with respect to this $L$). Then 
\begin{equation}\label{RD}
H^\infty_L (G)=\{f\in L^2 (G)\;:\; \int_G (1+L(x))^{2k} |f(x)|^2 dx <+\infty\}
\end{equation}
 satisfies  the properties (1) (2) (3) given in Definition \ref{def-ag}. 
Consequently, for  $G$ with $|\pi_0 (G)|<\infty$ and with the (RD) property,
there exists a subalgebra of $C^* (M,E)^G$, denoted $\mathcal{S}_{G}^\infty (M,E)$, which consists of integral operators, is dense and  holomorphically
closed in $C^*(M,E)^G$ and contains $\mathcal{S}_G^c (M,E)$ as a subalgebra.
\end{proposition}


\begin{proof}
The fact that $H^\infty_L (G)$ is not only contained in $C^*_r (G)$, via convolution, but in fact a subalgebra 
of it, follows from \cite{jolissaint-tams}. Hence  $H^\infty_L (G)$ satisfies the properties (1) and (2) given in Definition \ref{def-ag}. 
The fact that this subalgebra is holomorphically closed is proved as in 
\cite{jolissaint-k}.
The rest of the proposition then follows from Proposition \ref{lemma:HC-holoclosed}.
 \end{proof}

\noindent
\begin{example} Examples of Lie groups that satisfy property RD, and to which our theory applies, are:
\begin{itemize}
\item[1.] The abelian group $\RR^n$ satisfies (RD). In this case the algebra $H^\infty_L(\RR^n)$ associated to the length function defined by the Euclidean metric is the algebra of rapidly decaying functions on $\RR^n$.
\item[2.] Connected semisimple Lie groups satisfy property (RD), c.f.\ \cite{Chatterji-et-al-duke}, for example $G=SL(2,\RR)$. In this case the algebra $H^\infty_L(G)$ is closely related to Harish--Chandra's Schwartz  algebra $\mathcal{C}(G)$, see below.
\end{itemize}
\end{example}
\begin{remark}
We have just seen that if $G$ is semisimple then by choosing $\mathcal{A}(G)=H^\infty_L (G)$ 
we obtain a holomorphically closed
subalgebra $\mathcal{S}_{G}^\infty (M,E)\subset C^*(M,E)^G$. Notice that  there are other algebras that can be considered.
For example, we can consider as in  \cite{Hochs-Wang-HC}  the Harish-Chandra Schwartz  algebra $\mathcal{C}(G)\subset
C^*_r (G)$. This is a holomorphically closed subalgebra of $C^*_r (G)$, c.f.\ \cite{Lafforgue}, which is made of smooth functions
acting by convolution. 
The corresponding  algebra 
 $\mathcal{C}_G (M,E)\subset C^*(M,E)^G$
 is a subalgebra of  $C^*(M,E)^G$ with elements that are in fact smoothing operators.
One can prove, see  \cite[\S II.9]{Varadarajan}, that 
$ \mathcal{C}(G)\subset H^\infty_L (G) $ and thus, consequently, $\mathcal{C}_G (M,E)\subset\mathcal{S}_{G}^\infty (M,E)$.
Notice that Hochs and Wang have proved that the heat operator $\exp (-tD^2)$ is an element in 
$\mathcal{C}_{G} (M,E)$. Hence $\exp (-tD^2)\in \mathcal{S}_{G}^\infty (M,E)$.\\

\end{remark}

\section{Index classes}
%

\noindent
From now on we shall make constant use of the identification $\mathcal{A}^c_G (M,E)\equiv \mathcal{S}^c_G (M,E)$.

\subsection{The index class in $K_* (C^*(M,E)^G)$}We consider as before 
a closed  even-dimensional manifold $M$ with a proper cocompact action of $G$.
Let $D$ a $G$-equivariant odd $\ZZ_2$-graded Dirac operator.
Recall, first of all,  the classical Connes-Skandalis idempotent. Let $Q_\sigma$ be a $G$-equivariant parametrix
of $G$-compact support with remainders $S_\pm$; here the subscript $\sigma$ stands
for symbolic. Consider the $2\times 2$ matrix
\begin{equation}\label{CS-projector}
P_{\sigma}:= \left(\begin{array}{cc} S_{+}^2 & S_{+}  (I+S_{+}) Q\\ S_{-} D^+ &
I-S_{-}^2 \end{array} \right).
\end{equation}
This produces a class 
\begin{equation}\label{CS-class}
\operatorname{Ind_c} (D):= [P_{\sigma}] - [e_1]\in K_0 (\mathcal{A}_G^c (M,E))\;\;\text{with}\;\;e_1:=\left( \begin{array}{cc} 0 & 0 \\ 0&1
\end{array} \right)
\end{equation}
To understand  where this definition comes from, see for example \cite{CM}.\\
Recall now that $\mathcal{A}_G^c (M,E)\subset C^* (M,E)^G$.

\begin{definition}
The $C^*$-index associated to $D$ is the class $\Ind_{C^*(M,E)} (D)\in K_0 (C^* (M,E)^G)$ obtained by taking the 
image of the Connes-Skandalis
projector  in $K_0 (C^* (M,E)^G)$.\\ 
Unless absolutely necessary we shall denote this index class simply by $\Ind (D)$.
\end{definition}

\begin{remark}
If we are in the position of considering a dense holomorphically closed subalgebra $\mathcal{A}_G (M,E)$ 
of $C^* (M,E)^G$ as in the previous section, then we can equivalently take the image of the Connes-Skandalis  
projector in $K_0 (\mathcal{A}_G (M,E))$ (recall that, by construction, $\mathcal{A}_G^c (M,E)\subset \mathcal{A}_G (M,E)
\subset C^* (M,E)^G$).
For example, if $G$ satisfies (RD) and $|\pi_0 (G)|<\infty$, then we can take the $C^*$-index class as the image of the
Connes-Skandalis projector in $K_0 (\mathcal{S}^\infty_G (M,E))$.
\end{remark}

\begin{remark} 
There are other representatives of $\Ind (D)\in K_0 (C^* (M,E)^G)$ that can be of great
interest. \\ 
For example, as in Connes-Moscovici  
\cite{CM}, we can choose  the parametrix (which is not of $G$-compact support)
 $Q_{V}=
:= \frac{I-\exp(-\frac{1}{2} D^- D^+)}{D^- D^+} D^+$
obtaining
$I-Q_{V} D^+ = \exp(-\frac{1}{2} D^- D^+)$, $I-D^+ Q_{V} =  \exp(-\frac{1}{2} D^+ D^-)$.\\
This particular choice of parametrix produces
 the
idempotent
\begin{equation}
\label{cm-idempotent}
V_{D}=\left( \begin{array}{cc} e^{-D^- D^+} & e^{-\frac{1}{2}D^- D^+}
\left( \frac{I- e^{-D^- D^+}}{D^- D^+} \right) D^-\\
e^{-\frac{1}{2}D^+ D^-}D^+& I- e^{-D^+ D^-}
\end{array} \right)
\end{equation}
We call this the Connes-Moscovici idempotent. 
One can also consider the graph-projection $[e_{D}]-[e_1]\in  K_0 (C^* (M,E)^G)$
with  $e_D$  given by
\begin{equation}\label{graph-projector}
e_D=\left(\begin{array}{cc} (I+D^- D^+)^{-1} & (I+D^- D^+)^{-1} D^-\\ D^+   (I+D^- D^+)^{-1} &
D^+  (I+D^- D^+)^{-1} D^- \end{array} \right)\,.
\end{equation}
Finally, following
Moscovici and Wu \cite{MWU}, we can consider the 
projector 
\begin{equation}\label{MW-projector}
P(D):= \left(\begin{array}{cc} S_{+}^2 & S_{+}  (I+S_{+}) \mathcal{P}\\ S_{-} D^+ &
I-S_{-}^2 \end{array} \right).
\end{equation}
with $\mathcal{P}=\overline{u}(D^- D^+)D^-$, $S_+=I-\mathcal{P} D^+$, $S_- = I-D^+ \mathcal{P}$
and 
$\overline{u}(x):= u(x^2)$ with $u\in C^\infty(\RR)$ an even function with the property that $w(x)=1-x^2 u(x)$
is a Schwartz function and $w$ and $u$ have compactly supported Fourier transform. One proves easily that
$P(D)\in M_{2\times 2} (\mathcal{A}^c_G (M,E))$ (with the identity adjoined).
It is not difficult to prove  that
 $$ \Ind (D):= [P_{\sigma}] - [e_1]= [V_{D}]-[e_1]=[e_{D}]-[e_1]=[P(D)]-[e_1]\quad \text{in}\quad  K_0 (C^* (M,E)^G)\,.$$
The advantage of using the Connes-Moscovici projection, the graph projection or the Moscovici-Wu projection
is that Getzler rescaling can be used in order to prove the corresponding higher index formulae. This is crucial
if one wishes to pass, for example, to manifolds with boundary.
However, in this paper we shall concentrate solely on closed manifolds and will rather use the approach 
to the index theorem given in  \cite{PPT}; this employs  the algebraic index theorem in a fundamental way.

\end{remark}

\subsection{The index class in $K_\bullet (C^*_r (G))$}
There is a canonical Morita isomorphism $\mathcal{M}$ between $K_* (C^* (M,E)^G)$ and  $K_* (C^*_r (G))$.
This is clear
once we bear in mind that $C^* (M,E)^G$ is isomorphic to $\mathbb{K} (\mathcal{E})$; however, 
for reasons connected with the extension of cyclic cocycles, we want to be explicit about this
isomorphism. We assume the existence of a dense holomorphically closed subalgebra
$\mathcal{A}(G)\subset C^*_r (G)$ and follow Hochs-Wang  \cite{Hochs-Wang-HC}. 
Let $\mathcal{A}_G (M,E)$ be the dense holomorphically dense subalgebra
of $C^*(M,E)^G$ corresponding to $\mathcal{A}(G)$, as defined in Subsection \ref{subsect:holo}.
Define a partial trace map $\Tr_S: \mathcal{A}_G (M,E)\to \mathcal{A}(G)$
associated to the slice $S$ as follows: if $f\otimes k\in (\mathcal{A}(G))\hat{\otimes}\Psi^{-\infty}(S,E|_S))^{K\times K}$ then 
\[
\Tr_S (f\otimes k):= f \Tr (T_k)=f\int_S \tr k(s,s)ds,
\] 
with $T_k$ denoting the smoothing operator on $S$ defined by $k$ and
$\Tr (T_k)$ its functional analytic trace on $L^2 (S,E|_S)$. It is proved in  \cite{Hochs-Wang-HC} that this map
induces the Morita isomorphism $\mathcal{M}$ between $K_* (C^* (M,E)^G)$ and  $K_* (C^*_r (G))$.
We denote the image through $\mathcal{M}$ of the index class $\Ind (D)\in K_0 (C^*(M)^G)$ in the group
$K_0 (C^*_r (G))$ by $\Ind_{C^*_r (G)} (D)$. There are other, well-known descriptions of the latter index class:
one, following Kasparov, see \cite{kasparov-functor}, describes the $C^*_r (G)$-index class
as the difference of two finitely generated projective $C^*_r (G)$-modules, using the invertibility modulo $C^*_r G)$-compact
operators of (the bounded-tranform of) $D$; the other description
is via assembly  and $KK$-theory, as in \cite{BCH}.
All these descriptions of the class $\Ind_{C^*_r (G)} (D)\in K_0 (C^*_r (G)) $ are equivalent. See \cite{roe-comparing}
and \cite[Proposition 2.1]{PS-Stolz}.

\section{Cyclic cocycles and pairings with $K$-theory}
\subsection{Cyclic cohomology}
In this subsection we shortly review the basic complex computing cyclic cohomology. Let $A$ be a unital algebra. The space of reduced Hochschild cochains is defined as
\[
C^\bullet_{\rm red}(A):=\Hom_\mathbb{C}(A\otimes(A\slash\mathbb{C} 1)^\bullet,\mathbb{C})
\] 
and is equipped with the Hochschild differential 
$b: C^k_{\rm red}(A)\to C^{k+1}_{\rm red}(A)$ given by the standard formula
\[
b\tau(a_0,\ldots,a_{k+1}):=\sum_{i=0}^k(-1)^i\tau(a_0,\ldots,a_ia_{i+1},\ldots,a_k)+(-1)^{k+1}\tau(a_ka_0,\ldots,a_{k-1}).
\]
The cyclic bicomplex is given by
\[
\xymatrix{\ldots&\ldots&\ldots\\
C^2_{\rm red}(A)\ar[r]^B\ar[u]^b&C^1_{\rm red}(A)\ar[r]^B\ar[u]^b&C^0_{\rm red}(A)\ar[u]^b\\
C^1_{\rm red}(A)\ar[r]^B\ar[u]^b&C^0_{\rm red}(A)\ar[u]^b&\\
C^0_{\rm red}(A)\ar[u]^b
}
\]
where $B:C^k_{\rm red}(A)\to C^{k-1}_{\rm red}(A)$ denotes Connes' cyclic differential
\[
B\tau(a_0,\ldots,a_{k-1}):=\sum_{i=0}^{k-1}(-1)^{(k-1)i}\tau(1,a_i,\ldots,a_{k-1},a_0,\ldots, a_{i-1}).
\] 
We denote the total complex associated to this double complex by $CC^\bullet(A)$. When $A$ is not unital, we consider the unitization $\tilde{A}=A\oplus\mathbb{C}$, and 
compute cyclic cohomology from the complex $CC^\bullet(A):=CC^\bullet(\tilde{A})\slash CC^\bullet(\mathbb{C})$. 

Finally, let us close by mentioning that the structure underlying the definition of cyclic cohomology is that of a cocyclic object: this is a cosimplicial object 
$(X^\bullet,\partial^\bullet,\sigma^\bullet)$ equipped with an additional cyclic symmetry $t^n:X^n\to X^n$ of order $n+1$ satisfying well-known compatibility conditions with respect
to the coface operators $\partial$ and degeneracies $\sigma$, c.f.\ \cite{loday}. For the cyclic cohomology of an algebra the underlying cosimplicial object is given by 
$X^k=C^k(A)$ with coface and degeneracies controlling the Hochschild complex. The additional cyclic symmetry $t$ underlying cyclic cohomology is simply the operator which in 
degree $k$ cyclically permutes the $k+1$ entries in a cochain $\tau\in C^k(A)$.

\subsection{The van Est map in cyclic cohomology}
Let $G$ be a unimodular Lie group with $|\pi_0 (G)|<\infty$.
In this subsection we describe, following \cite{ppt1,PPT}, how to obtain cyclic cocycles from smooth group cocycles. In this, we can work with two algebras: $C^\infty_c(G)$, the 
convolution algebra of the group, and $\mathcal{A}^c_G(M)$, the algebra of invariant smoothing operators with cocompact support. In order to simplify the notation we take  the vector bundle $E$ to be the product bundle of rank 1.

We start with the following well-known remark: inspection of the differential \eqref{diff-gc}  shows that the cochain complex $(C^\bullet_{\rm diff}(G),\delta)$ computing smooth group 
cohomology $H^\bullet_{\rm diff}(G)$ comes from an underlying cosimplicial structure given by coface maps $\partial^i$ and codegeneracies $\sigma^j$ defined on the
vector space of homogeneous smooth group cochains $C^\bullet_{\rm diff}(G)$.
This simplicial vector space can be upgraded to a cocyclic one by the cyclic operator $t:C^\bullet\to C^\bullet$ given by
\[
(tf)(g_0,\ldots,g_k)=f(g_k,g_0,\ldots,g_{k-1}),\quad f\in C^k_{\rm diff}(G).
\]
As seen above, the Hochschild theory of this cocyclic complex is just the smooth group cohomology. The associated cyclic theory is given by $\bigoplus_{i\geq 0} H^{\bullet-2i}_{\rm diff}(G)$.

Let us now describe the associated cyclic cocycles on the convolution algebra $C^\infty_c(G)$. Instead of using the full complex of smooth group cochains, we shall restrict to the quasi-isomorphic subcomplex 
$C^\bullet_{{\rm diff},\lambda}(G)\subset C^\bullet_{\rm diff}(G)$ of {\em cyclic} cochains, i.e., cochains $c\in C^k_{\rm diff}(G)$ satisfying
\[
c(g_0,\ldots,g_k)=(-1)^k c(g_k,g_0,\ldots,g_{k-1}).
\]
Let $c\in C^k_{\rm diff}(G)$ be a smooth homogeneous group cochain. Define the 
cyclic cochain $\tau_c\in C^k(C^\infty_c(G))$ by 
\begin{equation}
\label{ccg}
\tau^G_c(a_0,\ldots,a_k):=\int_{G^{\times k}}c(e,g_1,g_1g_2,\ldots,g_1\cdots g_k)a_0((g_1\cdots g_k)^{-1})a_1(g_1)\cdots a_k(g_k)dg_1\cdots dg_k.
\end{equation}
Next up is the algebra $\mathcal{A}^c_G(M)$ of invariant smoothing operators with cocompact support. Again given a smooth homogeneous group cochain $c\in C^k_{\rm diff}(G)$,
we now define a cyclic cochain on this algebra by the formula
\begin{equation}
\label{cca}
\begin{split}
\tau^M_c(k_0,\ldots,k_n):=\int_{G^{\times k}}\int_{M^{\times (k+1)}}\chi(x_0)\cdots\chi(x_n)k_0(x_0,g_1x_1)\cdots k_n(x_n,(g_1\cdots &g_n)^{-1}x_0)\\
c(e,g_1,g_1g_2,\ldots,g_1\cdots  g_n)& dx_0\cdots dx_ndg_1\cdots dg_n.
\end{split}
\end{equation}
\begin{proposition} The following holds true:
\begin{itemize}
\item[$i)$]
The map $c\mapsto \tau^G_c$ defined above is a morphism of cochain complexes 
and therefore induces a map 
\[
\Psi_G:H^\bullet_{\rm diff}(G)\to HC^\bullet(C^\infty_c(G)).
\]
\item[$ii)$] The map $c\mapsto \tau^M_c$ defined above is a morphism of cocyclic complexes and therefore induces a map 
\[
\Psi_M:H^\bullet_{\rm diff}(G)\to HC^\bullet(\mathcal{A}^c_G(M)).
\]
\end{itemize}
\end{proposition}
\begin{proof}
Both of the statements are already known: for the first one, see \cite[\S 1.3]{ppt1}, and \cite[\S 2.2]{PPT} for the second.
\end{proof}
\begin{example}
In Example \ref{ex-gc} we discussed the smooth group $2$-cocycles for $G=\RR^2$, $G=SL(2,\RR)$, associated to the area forms of the homogeneous space $G\slash K$, equal to $\RR^2$ and $\mathbb{H}^2$ respectively.
Let us now consider the cyclic cocycles defined by these forms via the construction \eqref{ccg} above.  For $G=SL(2,\mathbb{R})$ this gives the following cyclic $2$-cocycle on $C^\infty_c(SL(2,\RR)$:
\[
\tau^{SL(2,\RR)}_\omega(f_0,f_1,f_2):=\int_{SL(2,\mathbb{R})}\int_{SL(2,\mathbb{R})} f_0((g_1g_2)^{-1})f_1(g_1)f_2(g_2){\rm Area}_{\mathbb{H}^2}({\Delta^2(\bar{e},\bar{g}_1,\bar{g}_2)})dg_1dg_2
\]
This is exactly the cyclic cocycle considered by Connes in \cite[\S 9]{Connes-ncdg}. For $G=\RR^2$ we get a cyclic $2$-cocycle on $C^\infty_c(\RR^2)$ (with convolution product) given by the same formula with the hyperbolic area replaced by the Euclidean area, 
and integrations being over $\RR^2$ instead of $SL(2,\RR)$, again considered in \cite[\S 9]{Connes-ncdg}. After Fourier transform $f\mapsto \hat{f}$ this cocycle takes the usual form 
\[
\tau_\omega(f_0,f_1,f_2)=\int_{\RR^2}\hat{f}_0d\hat{f}_1\wedge d\hat{f}_2,\quad \mbox{for}~ f_0,f_1,f_2\in C^\infty_c(\RR^2).
\]
\end{example}
\subsection{Extension properties}
In the previous subsection we have constructed cyclic cocycles $\tau^G_c$ on $C^\infty_c(G)$ and $\tau^M_c$ on $\mathcal{A}^c_G(M)$ from a homogeneous smooth group cocycle 
$c$. (Recall, once again, that for notational convenience we are taking $E$ to be the product rank 1 bundle.)
In \S \ref{subsect:holo} we have given sufficient conditions on $G$ ensuring that these algebras embed into holomorphically closed subalgebras $\mathcal{A}(G)$ and $\mathcal{A}_G(M)$ of the reduced group 
$C^*$-algebra and of the Roe algebra. Now we want to discuss the extension properties of these cocycles. 
Assume, quite generally, that we are given a subalgebra $\mathcal{A}(G)$ as in Definition
\ref{def-ag}, with associated  algebra of operators on $L^2 (M)$ denoted, as usual, as $\mathcal{A}_G(M)$. First we have:
\begin{proposition}\label{prop:2-cocycles}
Let $c\in C^k_{{\rm diff},\lambda}(G)$ be a smooth group cocycle. Then we have:
\[
\tau^G_c~\mbox{extends to}~\mathcal{A}(G) \Longleftrightarrow \tau^M_c~\mbox{extends to}~\mathcal{A}_G(M)
\]
\end{proposition}
\begin{proof}
Recall that the algebra $\mathcal{A}_G(M)$ is constructed from the choice of $\mathcal{A}(G)\subset C^*_r(G)$ by the slice theorem: an invariant kernel $k$ 
belongs to $\mathcal{A}_G(M)$ if the function
\[
\tilde{k}(g,s_1,s_2):=k(s_1,gs_2)
\]
belongs to $(\mathcal{A}(G)\hat{\otimes}\Psi^{-\infty}(S,E|_S))^{K\times K}$. These functions $\tilde{k}_i(g_i,x_i,x_{i+1})$,~$i=0,\ldots n-1$ and $\tilde{k}_n((g_1\cdots g_n)^{-1},x_n,x_0)$ are used in the formula \eqref{cca} for the cocycle $\tau^M_c$. Since the cut-off function $\chi$ has compact support, performing the integrations over $M$ in 
equation \eqref{cca}, we end up with the pairing of an element in $\mathcal{A}(G)^{\otimes(k+1)}$ with the group cocycle $c$ as defined in \eqref{ccg}. But then it is clear that $\tau^M_c$ is well-defined on $\mathcal{A}_G(M)$ if and only if $\tau^G_c$ is well-defined on $\mathcal{A}(G)$.
\end{proof}
For the following, recall from \S \ref{egc} the explicit form \eqref{cadf} of the van Est isomorphism mapping a closed invariant form $\alpha\in \Omega^k_{\rm inv}(G\slash K)$ to a smooth group cocycle $J(\alpha)\in C^k_{\rm diff}(G)$. For notational convenience,
we will drop the $J$ in the description of the associated cyclic cocycles, writing $\tau^G_\alpha$ and $\tau^M_\alpha$ instead of $\tau^G_{J(\alpha)}$ and $\tau^M_{J(\alpha)}$.
\begin{proposition}
Let $G$ be a Lie group with finitely many connected components and satisfying the rapid decay property (RD). Assume that $G/K$ is of non-positive sectional curvature. Then
the cocycle $\tau^G_\alpha$ associated to a closed invariant differential form $\alpha\in \Omega^k_{\rm inv}(G\slash K)$ extends continuously to $H^\infty_L (G)$. Consequently, the cyclic cocycles $\tau^M_\alpha$ extends to $\mathcal{S}^\infty_G (M)$. 
 \end{proposition}
\begin{proof}
Recall the definition of the smooth group cocycle $J(\alpha)\in C^k_{\rm diff}(G)$  defined in \eqref{cadf}, satisfying the polynomial estimates of Theorem \ref{cocycle:pol-est}. This, together with the rapid decay property of $G$, ensures we can follow the line of proof
of \cite[Prop. 6.5]{CM}, where the analogous extension property is proved for certain discrete groups.
To show that the cyclic cocycle $\tau_\alpha$ extends continuously to the algebra $H^\infty_L (G)$, we need to show that it is bounded with respect to the seminorm $\nu_k$ in \eqref{seminorm} defining the Fr\'ech\`et topology, for some $k\in\mathbb{N}$. 
Let $a_0,\ldots,a_k\in H^\infty_L (G)$, and write $\tilde{a}_0:=|a_0|$, $\tilde{a}_i(g):=(1+d(g))^k|a_i(g)|,~i=1,\ldots, k$. Then we can make the following estimate:
\begin{align*}
|\tau^G_\alpha(a_0,\ldots,a_k)|\leq& C \int_{G^{\times k}}(1+d(g_1))^{k}\cdots (1+d(g_k))^{k}|a_0((g_1\cdots g_k)^{-1})|\cdot |a_1(g_1)|\cdots |a_k(g_k)|dg_1\cdots dg_k\\
=&C (\tilde{a}_0*\ldots * \tilde{a}_k)(e)\\
\leq &C || \tilde{a}_0*\ldots * \tilde{a}_k||_{C^*_r(G)}\\
\leq & C||  \tilde{a}_0||_{C^*_r(G)}\cdots ||\tilde{a}_k||_{C^*_r(G)}\\
\leq & CD^{k+1} \nu_p(\tilde{a}_0)\cdots \nu_p(\tilde{a}_k) =CD^{k+1}\nu_{p+k}(a_0)\cdots \nu_{p+k}(a_k).
\end{align*}
In this computation we have used the fact that the Plancherel trace $a\mapsto a(e)$ on the convolution algebra has a continuous extension to $C^*_r(G)$, together with the rapid decay property: $||a||_{C^*_r(G)}\leq D||(1+d)^p a||_{L^2}$, for some $p$. Altogether, this proves the proposition.
\end{proof}
\subsection{Pairing with $K$-theory}
Cyclic cohomology was originally developed by Connes to pair with $K$-theory via the Chern character. Let us recall this construction: let $\tau=(\tau_0,\tau_2,\ldots,\tau_{2k})\in CC^{2k}(A)$ be a cyclic cocycle of 
degree $2k$ on a unital algebra $A$, and $[p]-[q]$ an element in $K_0(A)$ represented by idempotents $p,q\in M_N(A)$. The number
\[
\left<[p]-[q],\tau\right>:=\sum_{n=0}^{k}(-1)^n\frac{(2n)!}{n!}\left(\tau_{2n}\left(\tr(p-\frac{1}{2},p,\ldots,p)\right)-\tau_{2n}\left(\tr(q-\frac{1}{2},q,\ldots,q)\right)\right),
\]
where $\tr:M_N(A)^{\otimes(n+1)}\to A^{\otimes(n+1)}$ is the generalized matrix trace, is well-defined and depends only on the (periodic) cyclic cohomology class of $\tau$.
\begin{proposition}\label{prop-morita-cocycles}
Let $c$, $\mathcal{A} (G)$ and $\mathcal{A}_G (M)$ as in Proposition \ref{prop:2-cocycles}, and assume that $\tau^G_c$, 
and therefore $\tau^M_c$, extends. Then, under the Morita isomorphism $\mathcal{M}:K_0(C^*(M,E)^G)\stackrel{\cong}{\longrightarrow} K_0(C^*_r(G))$, we have the equality:
\[
\left<[p]-[q],\tau^M_c\right>=\left<\mathcal{M}([p]-[q]),\tau^G_c\right>.
\]
\end{proposition}
\begin{proof}
Recall that the isomorphism $\mathcal{M}:K(C^*(M,E)^G)\to K(C^*_r(G))$ is implemented by the partial trace map $\Tr_S:\mathcal{A}_G(M,E)\to \mathcal{A}(G)$ on the respective dense subalgebras. By the abstract Morita isomorphism $\mathcal{M}$, it suffices to consider a simple idempotent $e=e_1\otimes e_2\in M_n(\mathcal{A}_G(M,E))$
so that $\Tr_S(e)=\Tr_S(e_2)e_1$ yields an idempotent in $M_n(\mathcal{A}(G))$, where we have extended $\Tr_S$ to matrix algebras in the usual way by combining with the matrix trace.

Because we know that the cyclic cohomology class
of $\tilde{\tau}_c$ is independent of the choice of a cut-off function, the pairing with $K$-theory does not depend on this choice either so we can choose the family $\chi_\epsilon$ constructed in Lemma \ref{cut-off:family} and take the limit as $\epsilon\downarrow 0$:
\begin{align*}
\left<[e],\tau^M_c\right>=&\lim_{\epsilon\downarrow 0} \frac{(2k)!}{k!} \int_{G^{\times k}}\int_{M^{\times (k+1)}}\chi_\epsilon(x_0)\cdots\chi_\epsilon(x_n)e
(x_0,g_1x_1)\cdots e(x_n,(g_1\cdots g_n)^{-1}x_0)\\
&\hspace{5cm} c(e,g_1,g_1g_2,\ldots,g_1\cdots  g_n) dx_0\cdots dx_ndg_1\cdots dg_n,\\
=&
 \frac{(2k)!}{k!} \int_{G^{\times k}}\int_{S^{\times (k+1)}} e(x_0,g_1x_1)\cdots e(x_n,(g_1\cdots g_n)^{-1}x_0)\\
&\hspace{5cm} c(e,g_1,g_1g_2,\ldots,g_1\cdots  g_n) dx_0\cdots dx_ndg_1\cdots dg_n,\\
=& \frac{(2k)!}{k!} \Tr_S(e_2\cdots e_2) \int_{G^{\times k}} e_1(g_1)\cdots e_1((g_1\cdots g_n)^{-1})c(e,g_1,g_1g_2,\ldots,g_1\cdots  g_n)dg_1\cdots dg_n,\\
=&\left<[\mathcal{M}(e)],\tau^G_c\right>,
\end{align*}
where, to go to the last line, we have used the fact that $e_2^2=e_2$ is an idempotent. This completes the proof.
\end{proof}

\section{Higher $C^*$-indices and geometric applications}

\subsection{Higher $C^*$-indices and the index formula}
Let $M$ and $G$ be as above, with $M$ even dimensional. 
Hence $G$ is a unimodular Lie group with $|\pi_0 (G)|<\infty$.
(For the time being we do not put additional hypothesis on $G$.)
Let $E$ be an equivariant complex vector bundle.
Consider an odd $\ZZ_2$-graded  Dirac type operator $D$ acting on the sections of $E$.
We have then defined the compactly supported
index class $\Ind_c (D)\in K_0 (\mathcal{A}^c_G (M,E))$. Let $\alpha\in H^{{\rm even}}_{{\rm diff}} (G)$ and let $\Psi_M (\alpha)
\in HC^{{\rm even}} (\mathcal{A}_G^c (M,E))$ be 
the cyclic cohomology class corresponding to $\alpha$. 
We know that, in general, we have a pairing
\begin{equation}\label{pairing}
K_0 (\mathcal{A}^c_G (M,E))\times HC^{{\rm even}}(\mathcal{A}_G^c (M,E))\longrightarrow \CC
\end{equation}
We thus obtain, through $\Psi_M:H^\bullet_{\rm diff}(G)\to HC^\bullet(\mathcal{A}_G^c (M,E))
$,  a pairing 
\begin{equation}\label{pairing-bis}
K_0 (\mathcal{A}^c_G (M,E))\times H^{{\rm even}}_{{\rm diff}} (G) \longrightarrow \CC
\end{equation}
In particular, by  pairing $\Ind_c (D)\in K_0 (\mathcal{A}^c_G (M,E))$ with $\alpha\in H^{{\rm even}}_{{\rm diff}} (G)$
we  obtaining the {\it higher indices}
\begin{equation*}
\Ind_{c,\alpha} (D):=\left<\Ind_c (D),\Psi_M (\alpha)\right>\,,\quad \alpha\in  H^{{\rm even}}_{{\rm diff}} (G).
\end{equation*}
On the other hand, we can also take the image of $\alpha$ through the van Est map $\Phi_M: 
H^\bullet_{{\rm diff}} (G)\to H^\bullet_{{\rm inv}} (M)$; recall that this is nothing but the pull-back through the classifying map $\psi_M:
M\to G/K$ once we identify $H^\bullet_{{\rm diff}} (G)$ with $H^\bullet_{{\rm inv}} (G/K)$.
The following theorem is proved in \cite{PPT}:
\begin{theorem} [Pflaum-Posthuma-Tang, c.f.\ \cite{PPT}]
 Let $M$, $G$ and $D$ as above. In particular, $M$ is even dimensional.
Let $\alpha\in H^{{\rm even}}_{{\rm diff}} (G)$. Then the following identity holds true:
\begin{equation}\label{PPT-main}
 \Ind_{c,\alpha} (D)=\int_M \chi_M (m) \,{\rm AS} (M)\wedge \Phi_M (\alpha)
\end{equation}
with ${\rm AS}(M)$ the Atiyah-Singer integrand on $M$: ${\rm AS}(M):=\hat{A}(M,\nabla^M)\wedge\Ch^\prime (E,\nabla^E)$.\\
Equivalently, 
\begin{equation}\label{PPT-main-bis}
 \Ind_{c,\alpha} (D)=\int_M \chi_M (m) \,{\rm AS} (M)\wedge \psi_M^* (\alpha)
 \end{equation}
if we identify $H^\bullet_{{\rm diff}} (G)$
and 
$H^\bullet_{{\rm inv}} (G/K)$ via the van Est isomorphism, c.f.\ Remark \ref{van Est}.
\end{theorem}
We now make the fundamental assumption that $G$ satisfies the rapid decay property
and that $G/K$ is of non-positive sectional curvature. Let $\mathcal{S}^\infty_G (M,E)\subset C^* (M,E)^G$ be 
the dense
holomorphically closed subalgebra defined by the rapid decay algebra $H^\infty_L (G)\subset C^*_r (G)$. 
Thanks to the results of the previous Section we can extend the pairing \eqref{pairing-bis} to a pairing
\begin{equation}\label{pairing}
K_0 (\mathcal{S}^\infty_G (M,E))=K_0 (C^* (M,E)^G)\times H^{{\rm even}}_{{\rm diff}} (G)\longrightarrow \CC
\end{equation}
obtaining in this way the {\it higher $C^*$-indices} of $D$, denoted $\Ind_{\alpha} (D)$.
These numbers are well defined and can be computed by choosing a suitable representative
of the class $\Ind (D)\in K_0 (C^* (M,E)^G)$. Choosing the Connes-Skandalis projector we can apply again
the index formula of Pflaum-Posthuma-Tang, obtaing for each $\alpha\in H^{{\rm even}}_{{\rm diff}} (G)$ the $C^*$-index formula
\begin{equation}\label{PPT-main-C*}
 \Ind_{\alpha} (D)=\int_M  \chi_M (m) \,{\rm AS} (M)\wedge \Phi_M (\alpha)\,.
\end{equation}
Notice that  we also have a pairing
\begin{equation}\label{pairing-G}
K_0 (C^\infty_c (G))\times HC^{{\rm even}} (C^\infty_c (G)) \longrightarrow \CC
\end{equation}
and thus, through the homomorphism $\Psi_G: H^\bullet_{{\rm diff}}(G)\to HC^* (C^\infty_c (G))$,
a pairing 
\begin{equation}\label{pairing-G-bis}
K_0 (C^\infty_c (G))\times H^{{\rm even}}_{{\rm diff}}(G) \longrightarrow \CC
\end{equation}
According to the results of the previous section this pairing extends to a pairing
\begin{equation}\label{pairing-G-ter}
K_0 (C^*_r (G))\times H^{{\rm even}}_{{\rm diff}}(G) \longrightarrow \CC
\end{equation}
if $G$ satisfies (RD). In particular, we can define the $C^*_r (G)$-indeces $\Ind_{C^*_r(G),\alpha} (D)$
by pairing $\Ind_{C^*_r (G)} (D)\in K_0 (C^*_r (G))$ with $\alpha\in H^{{\rm even}}_{{\rm diff}}(G)$.
Moreover, from Proposition \ref{prop-morita-cocycles} we get the following equality:
\begin{equation}\label{indexG=index}
\left<\Ind(D),\Psi_M (\alpha)\right>=\left<\Ind_{C^*_r (G)} (D),\Psi_G (\alpha)\right>
\end{equation}
which means that 
\begin{equation}\label{indexG=index-bis}
\Ind_{C^*_r (G),\alpha} (D)=\Ind_\alpha (D)\quad\forall \alpha\in H^{{\rm even}}_{{\rm diff}}(G)
\end{equation}
and thus, thanks to \eqref{PPT-main-C*}, we can state the following fundamental result:

\begin{theorem}
Let $G$ be a Lie group satisfying the properties stated in the introduction: $|\pi_0 (G)|<\infty$, (RD) and $\underline{E}G$ of non-positive curvature.
Let $\alpha\in H^{{\rm even}}_{{\rm diff}} (G)$. Then there is a well-defined associated higher $C^*_r (G)$-index
$\Ind_{C^*_r (G),\alpha} (D)$
 and the following 
formula holds:
\begin{equation}\label{indexG=index-ter}
\Ind_{C^*_r (G),\alpha} (D)=\int_M  \chi_M (m) \,{\rm AS} (M)\wedge \Phi_M (\alpha)\,.
\end{equation}
Equivalently, if we identify $H^\bullet_{{\rm diff}}(G)$ and 
$H^\bullet_{{\rm inv}} (G/K)\equiv H^\bullet_{{\rm inv}} (\underline{E}G)$ via the van Est isomoprhism,  then
 $$\Ind_{C^*_r (G),\alpha} (D)=\int_M  \chi_M (m) \,{\rm AS} (M)\wedge \psi_M^* \alpha.$$
\end{theorem}
For $\alpha=1$, the associated cyclic cocycle \eqref{ccg} is just the Plancherel trace $\tau^G(f)=f(e)$ on $C^*_r(G)$, and the Theorem reduces to the $L^2$-index theorem first proved by Wang in \cite{Wang}.
Remark that in this case the trace extends to $C^*_r(G)$ without problems, so the assumptions on the curvature of $G/K$ and property (RD) are unnecessary.

\subsection{Higher signatures and their $G$-homotopy invariance}
Let $M$ and $N$ two orientable $G$-proper manifolds and let $f:M\to N$ be a $G$-homotopy equivalence. Let us denote by $D^{{\rm sign}}_M$ and 
$D^{{\rm sign}}_M$ the corresponding signature operators. Then, according to the main result in \cite{fukumoto2} we have that
\begin{equation}\label{HS}
\Ind_{C^*_r (G)} (D^{{\rm sign}}_M)=\Ind_{C^*_r (G)} (D^{{\rm sign}}_N)\quad\text{in}\quad K_0 (C^*_r (G))\,.
\end{equation}
Consequently, from \eqref{indexG=index-ter},
we obtain the following result, stated as item (i) in Theorem \ref{theo:main}
in the Introduction:

\begin{theorem}
Let $G$ be a Lie group satisfying the properties stated in the introduction: $|\pi_0 (G)|<\infty$, (RD) and $\underline{E}G$ of non-positive curvature.
Let $M$ and $N$ are two orientable  $G$-proper manifolds
and assume that there exists an orientation preserving $G$-homotopy equivalence 
between $M$ and $N$. Let us identify $H^\bullet_{{\rm diff}}(G)$ and 
$H^\bullet_{{\rm inv}} (G/K)\equiv H^\bullet_{{\rm inv}}(\underline{E}G)$ via the van Est isomorphism. 
Then. for each $\alpha\in  H^\bullet_{{\rm inv}}(\underline{E}G)$:
 $$\int_M  \chi_M (m) \,L(M)\wedge \psi_M^* \alpha= \int_N  \chi_N (n) \ L(N)\wedge \psi^*_N \alpha$$
 \end{theorem}
\begin{proof}
For even dimensional manifolds, this follows immediately from the previous discussion. For the odd-dimensional case we argue by suspension.
Thus, let $M$ be an orientable odd dimensional  $G$-proper manifold. We endow $M$ with a $G$-invariant riemannian metric $g_M$.
Consider $\RR$ and the natural action of  $\ZZ$ on it by translations (this is a free, proper and cocompact action).
 Taking the product of $M$ and $\RR$ we  obtaining the even dimensional $(G\times\ZZ)$-proper manifold $M\times\RR$;
 it has compact quotient equal to $M/G\times S^1$. We endow $M\times\RR$ with the $(G\times\ZZ)$-invariant metric
$g_M+dt^2$. 
 Consider the dual group $T^1:= {\rm Hom}(\ZZ,U(1))$.
 The signature operator on $M\times\RR$ defines an  index classe
in  the group $K_0 (C^* (M\times \RR)^{G\times \ZZ})$, which is isomorphic
to $K_0 (C^* (G)\hat{\otimes} C(T^1))$.
Consider the generator $d^\prime$ of $H^1 (\ZZ;\ZZ)\subset H^* (\ZZ;\CC)$ and let $d:=\frac{\sqrt{-1}}{2\pi} d^\prime\in
H^* (\ZZ;\CC)$. We know that $H^* (\ZZ;\CC)$ can be identified with $H^*_{\ZZ} (E\ZZ;\CC)$ and that $E\ZZ=\RR$;
we denote by $\Xi: H^* (\ZZ;\CC)\to H^*_{\ZZ} (\RR;\CC)=H^1 (S^1)$ this isomorphism.
Consider $\underline{E}G \times E\ZZ\equiv \underline{E}G \times \RR\equiv G/K\times \RR$.
To $\alpha\in H^{{\rm odd}}_{{\rm diff}} (G)\equiv H^{{\rm odd}}_{{\rm inv}} (\underline{E}G) \equiv
H^{{\rm odd}}_{{\rm inv}} (G/K)$ we associate $$\beta:= \alpha\otimes \Xi (d)\in H^{{\rm odd}}_{{\rm inv}} (G/K)\otimes H^1_{\ZZ}(\RR)=
 H^{{\rm odd}}_{{\rm inv}} (G/K)\otimes H^1 (S^1).$$
Now, on the one hand, we have natural homomorphisms
$$\Psi_{G\times\ZZ} : H^{{\rm odd}}_{{\rm inv}} (G/K)\otimes H^1 (S^1)\to HC^{{\rm even}} (C_c^\infty (G)\hat{\otimes} C^\infty (S^1))$$
and
$$\Psi_{M\times\RR} : H^{{\rm odd}}_{{\rm inv}} (G/K)\otimes H^1 (S^1)\to HC^{{\rm even}} (\mathcal{A}^c_{G\times\ZZ} (M\times \RR))$$
where we remark that $\mathcal{A}^c_{G\times\ZZ} (M\times \RR)=\mathcal{A}^c_{G}(M)\hat{\otimes} \mathcal{A}^c_{\ZZ} (\RR)$
and also that $\mathcal{A}^c_{G\times\ZZ} (M\times \RR)=C^*_c (M\times\RR)^{G\times\ZZ}$.
On the other hand the classifying map $\psi_M$ and the classifying map for the $\ZZ$-action on $\RR$ give together a smooth
$(G\times\ZZ)$-equivariant map $\psi_{M\times \RR} : M\times \RR \to G/K \times \RR$. We can apply the Pflaum-Posthuma-Tang
index theorem and obtain, for the signature operator,
$$\left<\Ind_{C^*_c(M\times\RR)^{G\times\ZZ}} (D_{M\times \RR}), \Psi_{M\times \RR} (\beta)\right>= \int_G \int_{S^1} \chi_M L(M\times\RR)
\psi^*_M (\alpha)\wedge \Xi(d)=\int_G \chi_M L(M)
\psi^*_M (\alpha)=\sigma (M,\alpha)\,.$$
If $G$ satisfies (RD), then this  formula remains true for the $C^*(M\times\RR)^{G\times\ZZ}$ index, because $\mathcal{S}^\infty_G (M)
\hat{\otimes} \mathcal{S}_\ZZ (\RR)$, with $\mathcal{S}_\ZZ (\RR)$ denoting the smooth $\ZZ$-invariant kernels of $\RR\times\RR$
of rapid polynomial decay, is a dense holomorphically closed subalgebra of $C^*(M\times\RR)^{G\times\ZZ}$ to which 
the pairing with $\Psi_{M\times \RR} (\beta)$ extends.
Consequently
$$\left<\Ind_{C^*(G)\hat{\otimes} C(S^1)} (D_{M\times \RR}), \Psi_{G\times \ZZ} (\beta)\right>=\sigma (M,\alpha)\,.$$
Now, if  $M$ and $N$ are G-homotopy equivalent, then $M\times \RR$ and $N\times \RR$
are $G\times\ZZ$ homotopy equivalent. Hence the corresponding signature
index classes in $K_0 (C^* (G)\hat{\otimes} C(T^1))$ are equal; thus
$$\left<\Ind_{C^*(G)\hat{\otimes} C(S^1)} (D_{M\times \RR}), \Psi_{G\times \ZZ} (\beta)\right>=
\left<\Ind_{C^*(G)\hat{\otimes} C(S^1)} (D_{N\times \RR}), \Psi_{G\times \ZZ} (\beta)\right>$$
This gives us $$\sigma(M,\alpha)=\sigma(N,\alpha)\,.$$ which is what we wanted to prove in odd dimension.
\end{proof}

\subsection{Higher $\widehat{A}$-genera and  G-metrics of positive scalar curvature}
Let $S$ be  compact smooth manifold with an action of a compact Lie group $K$. In general,
the existence of a $K$-invariant metric of positive scalar curvature on $S$ is a more refined property than the
existence of a positive scalar curvature metric on $S$; indeed, as shown by Berard-B\'ergery in 
 \cite{bb-psc}, averaging a positive scalar curvature metric on $S$ might destroy 
 the positivity of the  scalar curvature. For sufficient conditions on $K$ and on $S$ ensuring 
 the existence of such metrics see  \cite{lawson-yau-psc,hanke-symmetry}.

If $M$ is a $G$-proper manifold we can try to built a $G$-invariant positive scalar curvature 
metric on $M$ through a $K$-invariant positive scalar curvature  metric on the slice $S$. This is precisely what is
achieved in \cite{GMW}:

\begin{theorem} [Guo-Mathai-Wang, c.f.\ \cite{GMW}]
Let $G$ be an almost connected Lie group and let $K$ be a maximal compact subgroup of $G$. If $S$ is a compact manifold with a $K$-invariant metric of positive scalar curvature, then the $G$-proper manifold 
$G\times_K S$ admits a $G$-invariant metric of positive scalar curvature.
\end{theorem}

\noindent
This result shows that the space of positive scalar curvature $G$-metrics on a $G$-proper manifold
can be non-empty.

\smallskip
\noindent
We can ask for numerical obstructions to the existence of a positive scalar curvature $G$-metric.
Assume that $M$ has a G-equivariant spin structure and let $\eth$ be the associated
spin-Dirac operator. Then one can show, see again \cite{GMW}, that
\begin{equation}\label{vanishing-psc}
\Ind_{C^*_r (G)} (\eth)=0 \quad\text{in}\quad K_* (C^*_r G)\,.
\end{equation}

\noindent
The following result was stated as item (ii) in the main Theorem,
Theorem \ref{theo:main}, in the Introduction:

\begin{theorem}
Let $G$ be a Lie group satisfying the properties stated in the introduction: $|\pi_0 (G)|<\infty$, (RD) and $\underline{E}G$ of non-positive curvature.
Let $M$ be a $G$-proper manifold admitting a $G$-equivariant spin structure.
Let us identify $H^\bullet_{{\rm diff}}(G)$ and 
$H^\bullet_{{\rm inv}} (G/K)\equiv H^\bullet_{{\rm inv}}(\underline{E}G)$ via the van Est isomoprhism. 
If $M$ admits a $G$-invariant metric of positive scalar curvature, then 
$$\widehat{A}(M,\alpha):=\int_M  \chi_M (m) \,\widehat{A}(M)\wedge \psi_M^* \alpha=0$$
 for each $\alpha\in H^\bullet_{{\rm inv}}(\underline{E}G)$.
\end{theorem}

\begin{proof}
The even dimensional case follows directly from  our $C^*$-index formula and from 
\eqref{vanishing-psc}. In the odd dimensional case we argue by suspension, as we did
for the signature operator. It suffices to observe  that if 
$M$ is an odd dimensional   $G$-proper manifold admitting a $G$-equivariant spin structure
 and a $G$-invariant metric of positive scalar curvature $g_M$, then 
 $M\times \RR$ is  an  even dimensional $(G\times\ZZ)$-proper manifold 
 with a $(G\times\ZZ)$-equivariant spin structure and with 
 a  $(G\times\ZZ)$-invariant metric
$g_M+dt^2$ which is of positive scalar curvature too. Consequently, the analogue of 
\eqref{vanishing-psc} holds for the spin Dirac operator on $M\times\RR$ and so,
arguing as for the signature operator, we finally obtain
that $\widehat{A}(M,\alpha):=\int_M  \chi_M (m) \,\widehat{A}(M)\wedge \psi_M^* \alpha=0$
as required.
\end{proof}

{\small \bibliographystyle{plain}
\bibliography{rel-coverings}
}

\end{document}